\documentclass[11pt,reqno,twoside]{article}

\usepackage{fixltx2e} 

\usepackage{cmap} 

\usepackage[T1]{fontenc}
\usepackage[utf8]{inputenc}
\usepackage{graphicx}
\usepackage{placeins}
\usepackage{enumerate}
\usepackage{algcompatible}

\usepackage{subcaption}

\usepackage{verbatim}

\usepackage{setspace}

\let\counterwithin\relax  
\usepackage{lmodern} 
\usepackage[scale=0.88]{tgheros} 

\usepackage{bm} 

\usepackage{bbold}

\usepackage{amsmath,amsbsy,amsgen,amscd,amsthm,amsfonts,amssymb} 

\usepackage[centering,top=1.5in,bottom=1.2in,left=0.75in,right=0.75in]{geometry}

\usepackage{titling}
\graphicspath{ {./images/} }

\usepackage[sf,bf,compact]{titlesec}

\usepackage{booktabs,longtable,tabu} 
\usepackage[font=small,margin=12pt,labelfont={sf,bf},labelsep={space}]{caption}

\usepackage[usenames,dvipsnames]{xcolor}

\definecolor{dark-gray}{gray}{0.3}
\definecolor{dkgray}{rgb}{.4,.4,.4}
\definecolor{dkblue}{rgb}{0,0,.5}
\definecolor{medblue}{rgb}{0,0,.75}
\definecolor{rust}{rgb}{0.5,0.1,0.1}

\usepackage{url}
\usepackage[colorlinks=true]{hyperref}
\hypersetup{linkcolor=dkblue}    
\hypersetup{citecolor=rust}      
\hypersetup{urlcolor=rust}     

\usepackage[final]{microtype} 

\newtheoremstyle{myThm} 
    {\topsep}                    
    {\topsep}                    
    {\itshape}                   
    {}                           
    {\sffamily\bfseries}                   
    {.}                          
    {.5em}                       
    {}  

\newtheoremstyle{myRem} 
    {\topsep}                    
    {\topsep}                    
    {}                   
    {}                           
    {\sffamily}                   
    {.}                          
    {.5em}                       
    {}  

\newtheoremstyle{myDef} 
    {\topsep}                    
    {\topsep}                    
    {}                   
    {}                           
    {\sffamily\bfseries}                   
    {.}                          
    {.5em}                       
    {}  

\theoremstyle{myThm}
\newtheorem{theorem}{Theorem}[section]
\newtheorem{lemma}[theorem]{Lemma}

\theoremstyle{myRem}
\newtheorem{remark}[theorem]{Remark}

\theoremstyle{myDef}

\usepackage{fancyhdr}
\usepackage{nopageno} 
\usepackage{algorithm}
\usepackage{algpseudocode}

\let\originalleft\left
\let\originalright\right
\renewcommand{\left}{\mathopen{}\mathclose\bgroup\originalleft}
\renewcommand{\right}{\aftergroup\egroup\originalright}
\renewcommand{\div}{\text{div}}

\usepackage{mathtools}
\mathtoolsset{centercolon}  

\renewcommand{\phi}{\varphi}
\newcommand{\eps}{\varepsilon}
\newcommand{\tx}{\tilde{x}}

\providecommand{\mathbbm}{\mathbb} 

\newcommand{\R}{\mathbbm{R}}

\newcommand{\OO}{\mathcal{O}}
\newcommand{\G}{\mathcal{G}}
\newcommand{\F}{\mathcal{F}}

\renewcommand{\L}{\mathcal{L}}

\definecolor{mygreen}{rgb}{0.1,0.75,0.2}

\newcommand{\sam}{n}

\newcommand{\M}{\mathcal{M}}

\newcommand{\kmin}{\kappa_{\mbox {\tiny{\rm min}}}}

\newcommand{\dtv}{d_{\mbox {\tiny{\rm TV}}}}

\usepackage[]{algorithm}

\usepackage{graphicx}

\usepackage{authblk}
\usepackage[square,numbers]{natbib}
\usepackage{chngcntr}
\usepackage{mathrsfs}
\counterwithin{table}{section}
\counterwithin{algorithm}{section}

\title{Kernel  Methods for Bayesian Elliptic Inverse Problems on Manifolds} 
\author{John Harlim $^{\dagger}$,  Daniel Sanz-Alonso $^{*}$,  Ruiyi Yang $^*$}
\date{}

\vspace{.25in}

\makeatletter\@addtoreset{section}{part}\makeatother%
\numberwithin{equation}{section}

\newcommand{\upperRomannumeral}[1]{\uppercase\expandafter{\romannumeral#1}}

\renewcommand{\hat}{\widehat}

\begin{document}
\maketitle 


\abstract{This paper investigates the formulation and implementation of Bayesian inverse problems to learn input parameters of partial differential equations (PDEs) defined on manifolds. Specifically,  we study  the inverse problem of determining the diffusion coefficient of a second-order elliptic PDE on a closed manifold from noisy measurements of the solution. Inspired by manifold learning techniques, we approximate the elliptic differential operator with a kernel-based integral operator that can be discretized via Monte-Carlo without reference to the Riemannian metric. The resulting computational method is mesh-free and easy to implement, and can be applied without full knowledge of the underlying manifold, provided that a point cloud of manifold samples is available. We adopt a Bayesian perspective to the inverse problem, and establish an upper-bound on the total variation distance between the true posterior and an approximate posterior defined with the kernel forward map. Supporting numerical results show the effectiveness of the proposed methodology.
\let\thefootnote\relax\footnote{$^{\dagger}$ The Pennsylvania State University, Department of Mathematics, Department of Meteorology \& Atmospheric Science, Institute for CyberScience.}}
\let\thefootnote\relax\footnote{$^*$ University of Chicago, Department of Statistics.}

\section{Introduction}\label{sec:introduction}

Partial Differential Equations (PDEs) on manifolds are used to model a variety of physical and biological phenomena including pattern formation on biological surfaces, phase separation in bio-membranes, tumor growth, and surfactants on fluid interfaces \cite{eilks2008numerical,elliott2010surface,chaplain2003mathematical,xu2006level}. In this paper we focus on the inversion of \emph{elliptic} PDEs for two main reasons. First, elliptic PDEs are ubiquitous in applications and they are used, for instance, as simplified models for groundwaterflow and oil reservoir simulation. The need to specify uncertain input parameters of these models leads naturally to the inverse problem of determining the permeability from the pressure under a Darcy model of flow in a porous medium \cite{mclaughlin1996reassessment,lorentzen2012history,iglesias2007representer,ping2014history}. Second, elliptic models are widely used to test algorithms for forward propagation of uncertainty  \cite{frauenfelder2005finite,cohen2011analytic,babuska2004galerkin}  and Bayesian inversion \cite{AS10,cotter2013,trillos2016bayesian}. Despite the applied importance of elliptic inversion, 
the manifold setting that we consider is largely unexplored and may allow for more realistic modelling in applications. For example, the variables of interest in the groundwaterflow problem  may not be confined to a \emph{flat} two-dimensional domain and knowledge of the underlying flow surface may be limited to a point cloud of landmark locations.


The aim of this paper is to study the formulation and implementation of Bayesian inverse problems to learn input parameters of PDEs defined on manifolds. Specifically, we study the inverse problem of recovering the diffusion coefficient of a second-order, divergence-form elliptic equation, given noisy measurements of the solution. While our interest lies in solving the inverse problem, most of our efforts are devoted to studying the approximation of the \emph{forward map} (the operator that takes the input parameter to the solution of the PDE). Several techniques to approximate the forward map have been proposed in the extensive literature on numerical methods for PDEs on manifolds. For example, finite element methods \cite{dziuk2013finite,camacho,bonito2016high}, level-set methods \cite{bertalmio2001variational,memoli2004implicit}, closest point methods \cite{ruuth2008simple}, or mesh-free radial basis function methods \cite{piret2012orthogonal}. The implementation detail of each of the existing methods is different, but a unifying theme is the need to have a representation of the manifold in order to approximate the differential operator. Unfortunately, these approaches are difficult to implement when one only has access to an unstructured point cloud of manifold samples and meshing is challenging, or when the dimension of the ambient space is large but the manifold dimension is moderate. 



In this paper, we avoid the problems associated with the representation of the manifold by directly approximating the differential operator in the forward map with an appropriate kernel integral operator. With a consistent kernel approximation to the differential operator on the manifold, the numerical implementation can be performed naturally by discretizing the corresponding integral operator on  a \emph{point cloud} of manifold samples without further knowledge of the underlying manifold or its Riemannian metric. Building on this construction, we propose a fully discrete, mesh-free approach to the numerical solution of Bayesian inverse problems on point clouds. The idea of facilitating the discretization of PDEs on manifolds by an integral equation approximation can  also  be found in the recent papers \cite{li2016convergent,li2017point,gilani2019approximating}, all of which build on manifold learning techniques and analyses. Our perspective in this paper and in \cite{gilani2019approximating,bs:16} is in contrast to the one taken in  \cite{belkin2004semi,belkin2005towards,coifman2005geometric,coifman2006diffusion,berry2016variable}. Rather than identifying the limiting continuum operator of different normalizations of graph-Laplacians, our interest is to define a suitable kernel to approximate a given anisotropic diffusion operator on the underlying manifold. 

We adopt a Bayesian approach to the inverse problem \cite{kaipio2006statistical,calvetti2007introduction,AS10,sanzstuarttaeb},  where we set a \emph{prior} distribution on the unknown PDE input parameter, and condition on observed data to find its \emph{posterior} distribution. The Bayesian approach is largely motivated by the following advantages. First, the posterior covariance and posterior confidence intervals may be used to quantify the uncertainty in the parameter reconstruction. Second, the Bayesian formulation leads to a well-posed inverse problem \cite{marzouk2009stochastic,AS10,sanzstuarttaeb} by which a small perturbation on the data, the prior distribution, or the forward map leads to a small perturbation in the posterior solution. Our main theoretical result is an example of the well-posedness of the Bayesian formulation: we deduce a total variation error bound between the true posterior distribution and its kernel-based approximation from a new error bound between the forward map and its kernel approximation. The new forward map error bound, with a dependence on the diffusion coefficient, builds on existing results on the point-wise convergence kernel approximations to elliptic operators \cite{coifman2006diffusion,berry2016variable}. 

The advantages of the Bayesian approach outlined above come with a cost: the need to specify a prior distribution on the unknown. The choice of prior is crucial as it determines the support of the posterior, but unfortunately this choice is often only guided by ad-hoc and computational considerations. In this paper we consider a two-parameter family of log-Gaussian field priors on manifolds, defining the covariance through the Laplace Beltrami operator on the manifold \cite{lindgren2011explicit,dunlop2017hierarchical}.
The two prior parameters allow the specification of the smoothness and length-scale of prior (and hence posterior) draws, and the length-scale  may be learned from data using a hierarchical approach as detailed in our numerical experiments. In addition to their flexibility, a further advantage of our choice of priors is that they allow the infusion of geometric information from the manifold on the reconstructed input by expressing it as a random combination of the first eigenfunctions of the Laplacian. Moreover, in the absence of a full representation of the manifold, these priors can be consistently discretized using a graph-Laplacian \cite{von2007tutorial,garcia2018continuum}. We refer to    \cite{bertozzi2018uncertainty,gao2019gaussian,garcia2018continuum,trillos2017consistency}  
for recent applications and references on Gaussian processes on manifolds. From a computational viewpoint, the use of log-Gaussian priors and the existence of a continuum limit facilitate the design of MCMC algorithms that scale well with the size of the point cloud, as shown in a linear inverse problem in \cite{trillos2017consistency}. In this regard, a simple but powerful idea is to use a proposal kernel that satisfies detailed balance with respect to the prior \cite{cotter2010approximation}. 

\paragraph{Outline and Main Contributions}
We close this introduction with an outline of the rest of the paper, summarizing the main contributions of each section.
\begin{itemize}
\item  In Section \ref{sec:problemdescription} we give a brief introduction to the Bayesian formulation of inverse problems, and formulate the problem on a manifold.  The main novel contributions of this section are: i) to introduce a kernel-based approximation to the forward map and an associated approximation to the posterior in the continuum (Subsection \ref{ssec:kernelellipticip}); and ii) to employ the kernel approximations to formulate a Bayesian solution to the inverse problems on point clouds (Subsection  \ref{ssec:pointcloudellipticip}). These kernel and point cloud approximations are inspired by manifold learning and data analysis techniques. 
\item  Section \ref{sec:theory} contains the main theoretical contributions of this paper. 
Theorem \ref{thm:1} gives an error bound between the true and kernel-based forward maps, and Theorem \ref{thm:2} establishes a bound on the total variation distance between the posterior and its kernel-based approximation. The main novelty of Theorem \ref{thm:1}  is to generalize the analysis in \cite{coifman2006diffusion} to account for anisotropic diffusion and, more importantly, to explicitly track  the dependence of the diffusion coefficient in the error bounds. Understanding this dependency is necessary in order to guarantee the closeness of the true and approximate posterior distributions shown in Theorem \ref{thm:2}. 
\item In Section \ref{sec:numerics} we discuss the practical implementation of the methods, provide guidelines for the choice of tuning parameters and for the posterior sampling, and conduct three numerical experiments of increasing difficulty to illustrate the applicability of our approach. We also consider in Subsection \ref{sec:HB} a hierarchical formulation of the inverse problem, where the prior length-scale is learned from the data, when in absence of such knowledge.
\item We close in Section \ref{sec:conclusions} with conclusions and open directions for research that stem from our work. 
\end{itemize}

\paragraph{Notation and Setting} Throughout this paper $\mathcal{M}$ will denote  an $m$-dimensional smooth Riemannian manifold  embedded in $\R^d.$ We will denote by  $\mathcal{C}^k:=\mathcal{C}^k(\mathcal{M})$ the space of $k$-times differentiable functions on $\mathcal{M}$ and by $\mathcal{C}^{k,\alpha}:=\mathcal{C}^{k,\alpha}(\mathcal{M})$ the space of $k$-times differentiable functions whose $k$-th partial derivatives are H\"older continuous with exponent $\alpha$. 
We will assume that the manifold $\M$ is compact and has no boundary, thus avoiding the technicalities necessary to deal with boundary conditions. The theoretical and computational investigation of point cloud approximation to PDEs supplemented with boundary conditions is a topic of current research \cite{li2016convergent,li2017point,gilani2019approximating}. Due to the lack of boundary conditions, the elliptic problem that we consider is unique up to a constant. We will enforce uniqueness  by working on the space $L^{2}_0:=L^2_0(\mathcal{M})$ of mean-zero square integrable functions on $\mathcal{M}$. 

\section{Bayesian Inverse Problems on Manifolds and Point Clouds}\label{sec:problemdescription}
We start in Subsection \ref{ssec:ellipticip} by recalling the Bayesian formulation of elliptic inverse problems in a given manifold. We then introduce in Subsection \ref{ssec:kernelellipticip} a kernel-based approximation to the forward map and a corresponding approximation to the posterior, both of which will be analyzed in Section \ref{sec:theory}. Finally in Subsection \ref{ssec:pointcloudellipticip} we introduce a point cloud approximation of the kernel forward map leading to a formulation of the elliptic problem on point clouds without reference to the underlying manifold. We will investigate numerically the implementation of elliptic Bayesian inverse problems on point clouds in Section \ref{sec:numerics}.

\subsection{Bayesian Elliptic Inverse Problems on Manifolds}\label{ssec:ellipticip}
We consider the elliptic equation
\begin{equation} \label{eq:pde}
\L^\kappa u := -\div( \kappa \nabla u)  = f, \quad \quad x\in \M.
\end{equation} 
Here and throughout, the differential operators are defined with respect to the Riemannian metric inherited by $\M$ from $\R^d.$ We are interested in the inverse problem of determining the diffusion coefficient $\kappa$ from noisy measurements of $u$ of the form 
\begin{equation}\label{eq:obs}
y = \mathcal{D}(u) + \eta,
\end{equation}
where the \emph{observation map} $\mathcal{D}: L^2 \to \R^J$  will be assumed to be known. Examples of observation maps will be discussed in Subsetion \ref{sec:obsmap}. We adopt a Bayesian perspective to the inverse problem, described succinctly in what follows;  we refer to \cite{kaipio2006statistical,AS10,sanzstuarttaeb} for a more detailed account. In short, the Bayesian formulation of inverse problems involves specifying a \emph{prior} distribution $\pi$ for the unknown PDE input $\kappa$ and a distribution for the observation noise $\eta.$ Once these distributions have been specified, the solution to the inverse problem is the \emph{posterior} distribution of the variable $\kappa$ conditioned on the observed data $y.$ For simplicity of exposition, we will assume throughout that the observation noise is centered and Gaussian, $\eta \sim \mathcal{N}(0,\Gamma)$ for given positive definite $\Gamma \in \R^{J\times J}.$ 

Writing $\kappa = e^\theta$, we take a Gaussian prior for $\theta$ supported on a Banach space $\mathcal{B}$. A specific form of  prior, widely used in applications, will be described in Subsection \ref{sec:ctmprior}.
 Our assumptions on $\kappa$ and $f$ in Section \ref{sec:theory} will guarantee the existence of a unique solution to equation \eqref{eq:pde} in the space $L_0^2$ of mean-zero square intergrable functions on $\M.$ This, in turn, allows us to define a \emph{forward map} $\mathcal{F}:\theta \in \mathcal{B} \mapsto u \in L_0^2$. Provided that the map $\G := \mathcal{D}\circ \F: \mathcal{B} \to \R^J$ is measurable and that the prior is supported on $\mathcal{B}$, the posterior $\pi^y$ can be written as a change of measure with respect to the prior
\begin{align}\label{eq:posterior}
\frac{d\pi^y}{d\pi}(\theta) \propto \exp\left( -\frac{1}{2} |y-\mathcal{G}(\theta)|^2_{\Gamma} \right), 
\end{align}
with $| \cdot|_\Gamma^2 := \langle \cdot ,\Gamma^{-1} \cdot \rangle.$
Equation \eqref{eq:posterior} shows that the posterior distribution $\pi^y$ is defined by reweighting the prior, favoring  unknowns $\theta$ that produce a good match with the data $y$  through a likelihood function (the right-hand term), implied by equation \eqref{eq:obs} and the assumed Gaussian distribution of the noise $\eta.$

For our theoretical results in Section \ref{sec:theory} we will take $\mathcal{B}=\mathcal{C}^4$ and assume that $f \in \mathcal{C}^{3,\alpha}$, for $0 <\alpha <1.$  These assumptions guarantee \cite{gilbarg2015elliptic} that almost surely with respect to the prior, the diffusion coefficient $\kappa$ is uniformly elliptic, and  the unique solution of equation \eqref{eq:pde} in $L_0^2$ lives in $\mathcal{C}^{5,\alpha}$, allowing us to establish a stability result for an approximation of the forward map. We believe, however, that these strong regularity conditions can be relaxed.


\subsubsection{Mat\'ern-type Prior} \label{sec:ctmprior}
Here we describe a choice of prior that is widely used in applications in the geophysical sciences and spatial statistics \cite{stein2012interpolation}; in Subsection \ref{ssec:graphprior} we will introduce  a point cloud approximation to this prior used in our numerical experiments in Section \ref{sec:numerics}.
Since equation \eqref{eq:posterior} implies that the support of the prior determines the support of the posterior, it should both capture the geometry of the manifold $\M$ and have enough expressivity to include a wide class of functions. This motivates to choose the prior from a flexible two-parameter family of Gaussian measures on $L^2$. Precisely, we will consider priors of the form 
\begin{align}
	\pi=\mathcal{N}(0,C_{\tau,s}), \quad C_{\tau,s}=c(\tau)(\tau I+ \Delta_{\M})^{-s}, \label{eq:ctmprior}
\end{align}
where $\Delta_{\M}:=-\text{div}(\nabla \cdot)$ is the Laplace-Beltrami operator on $\M,$ $\tau>0, s>\frac{m}{2}$ are two free parameters, whose intuitive interpretation will be given below, and $c(\tau)$ is a normalizing constant. Let $\{(\lambda_i,\phi_i)\}_{i=1}^{\infty}$ be eigenvalue-eigenvector pairs for $\Delta_{\M}$ with $\lambda_i$'s increasing. Then by the Karhunen-Lo\'eve expansion, random samples of $\pi$ admit a  series expansion
\begin{align}
v=c(\tau)^{1/2}\sum_{i=1}^{\infty} (\tau + \lambda_i)^{-s/2} \xi_i \phi_i, \label{eq:sample}
\end{align}
where $\xi_i \overset{\text{i.i.d.}}{\sim} \mathcal{N}(0,1)$. The eigenfunctions of the Laplacian contain geometric information on the underlying manifold, and therefore constitute a natural basis for functions on the manifold. By Weyl's law, $\lambda_i \asymp i^{2/m}$ and so the  requirement $s>\frac{m}{2}$ is to ensure that samples from $\pi$ belong to $L^2$ almost surely. Moreover, the parameter $s$ controls the rate of decay of the coefficients and  hence characterizes the regularity of the samples.  The role of $\tau$ is more delicate.  If we write the coefficients as $v_i:=(\tau+\lambda_i)^{-s/2}=\tau^{-s/2}(1+\frac{\lambda_i}{\tau})^{-s/2}$, then we can see that the $v_i$'s  will converge to 0 quickly for $\lambda_i$'s that are much larger than $\tau$. In particular, the only significant $v_i$'s are those where the corresponding $\lambda_i$'s  are on the same order of $\tau$ and hence $\tau$ determines the significant basis functions in the expansion \eqref{eq:sample}. Since the eigenfunctions $\{\phi_i\}_{i=1}^{\infty}$ represent increasing frequencies, $\tau$ can be interpreted as a length-scale parameter. It can be seen from \eqref{eq:sample} that $\tau$ also affects the amplitude of the samples and this motivates to choose the normalizing constant so that $v$ has a fixed variance, which we set to be 1:
\begin{align}
	c(\tau) = \frac{1}{\sum_{i=1}^{\infty} (\tau + \lambda_i)^{-s} }. \label{eq:normalize}
\end{align}

Such priors are widely used when $\M$ is a domain in a Euclidean space and are related to the Whittle-Mat\'ern distributions \cite{dunlop2017hierarchical}. In \cite{lindgren2011explicit} the authors also considered their extension to manifolds. It can be shown that for $s$ large enough, samples from $\pi$ belong to $\mathcal{C}^k$ almost surely. And since the embedding of $\mathcal{C}^k$ into $L^2$ is continuous, the restriction of $\pi$ to $\mathcal{C}^k$ is again a Gaussian measure \cite{bogachev1998gaussian}. Hence for our purpose, we choose $s$ so that $\pi$ is a Gaussian measure on $\mathcal{C}^4$. In particular we will need later in Section \ref{sec:theory} the result from Fernique's theorem \cite{fernique1970integrabilite} that there exists $\alpha>0$ such that 
\begin{align*}
\int_{\mathcal{B}} \exp\left(\alpha \|\theta\|^2_{\mathcal{C}^4}\right) d\pi(\theta)<\infty.
\end{align*}

\begin{remark} \label{rmk:tau}
Choosing a prior with parameter $\tau$ that is far from the true length-scale of the unknown parameter would lead to poor Bayesian inversion. This can be problematic if such prior knowledge is not available, but may be at least partially alleviated by considering a hierarchical formulation specifying a joint prior on both $\tau$ and $\theta$, so that the length-scale is learned from data simultaneously with the unknown $\theta$; implementation details of the hierarchical formulation will be given in Subsection \ref{sec:HB}. 
 \end{remark}

\subsubsection{Observation Maps} \label{sec:obsmap}
Here we give two examples of observation maps that we shall consider. For theoretical considerations, we assume that the observation map is of the form $\mathcal{D}(u)=\left(\ell_1(u),\ldots,\ell_J(u)\right)^T$, where each $\ell_j$ is a bounded linear functional on $L^2$. A widely used example is the smoothed observation at a point $x_j$: $\ell_j(u) = \int K(x_j,x) u(x) dV(x)$, where $K$ is a kernel such as the Gaussian kernel \cite{bigoni2019greedy}; this type of observations arise in practice when the data is gathered from a collection of spatially distributed sensors located in the vicinity of landmark points $x_j$.  Equally common is the pointwise evaluation \cite{trillos2016bayesian,dashti2013bayesian}: $\ell_j=u(x_j)$. Notice that pointwise evaluation is not a bounded linear functional on $L^2$ but can be approximated by smoothed observation arbitrarily well for continuous $u$'s. We remark that the boundedness assumption of $\ell_j$ is only a technical one and for the numerical experiments in Section \ref{sec:numerics} we will consider only pointwise evaluations.

\subsection{Kernel Approximation of the Forward and Inverse Problem}\label{ssec:kernelellipticip} 
In this subsection we introduce a kernel approximation $\L^\kappa_\eps$ to the operator $\L^\kappa.$ Instead of  directly discretizing the differential operators on $\M$, the new kernel operator is defined by an integral that can be discretized by Monte-Carlo integration as described in the next subsection. Our kernel approximation is inspired by the following construction and result found in \cite{coifman2006diffusion}.

Let
\begin{align*}
    G_{\eps}u(x) :=\eps^{-\frac{m}{2}} \int_{\mathcal{M}} h\left(\frac{|x-\tx|^2}{\eps} \right) u(\tx) dV(\tx),\quad h(z) :=\frac{1}{\sqrt{4\pi}} e^{-\frac{z}{4}},
\end{align*}
where $dV$ denotes the volume form inherited by $\mathcal{M}$ from the ambient space $\mathbb{R}^d$. Then Lemma 8 in \cite{coifman2006diffusion} shows that, for $u$ sufficiently smooth,
\begin{align}\label{eq:geps}
G_{\eps} u(x)= u(x)+ \eps \big( \omega u(x) -\Delta_{\M} u(x) \big) +O(\eps^2), \quad \quad x\in \M.
\end{align}
Here, $\Delta_{\M}:= -\div(\nabla \cdot),$ and  $\omega$ is a function that depends only on the embedding of $\M.$
Now, note that by direct calculation
\begin{align}\label{eq:directcalculation}
\L^\kappa u := -\text{div}(\kappa \nabla u) = \sqrt{\kappa} \left[u\Delta_{\M} \sqrt{\kappa}-\Delta_{\M}(u\sqrt{\kappa})\right],
\end{align}
and that the expansion \eqref{eq:geps} for $\sqrt{\kappa}$ and $u \sqrt{\kappa}$ yields
\begin{align*}
u G_{\eps} \sqrt{\kappa}  &= u \sqrt{\kappa}  + \eps \big( \omega u\sqrt{\kappa}  -u\Delta_{\M} \sqrt{\kappa}  \big) +O(\eps^2),  \\
G_{\eps} (u\sqrt{\kappa}) &= u\sqrt{\kappa}  + \eps \big( \omega u\sqrt{\kappa}  -\Delta_{\M} (u \sqrt{\kappa})  \big) +O(\eps^2).
\end{align*}
Substracting both equations and using \eqref{eq:directcalculation} gives that 
\begin{align*}
	G_{\eps}(u\sqrt{\kappa}) - uG_{\eps} \sqrt{\kappa}&=  \eps\left[u\Delta_{\M} \sqrt{\kappa}-\Delta_{\M}(u\sqrt{\kappa})\right] +O(\eps^2) = \frac{\eps}{\sqrt{\kappa}} \L^{\kappa} u +O(\eps^2). 
\end{align*}
This equation motivates the following definition of the integral operator $\L^{\kappa}_{\eps}$
\begin{align*}
\L^\kappa_{\eps} u(x)& :=\frac{\sqrt{\kappa (x)}}{\eps} \left[ u(x)G_{\eps}\sqrt{\kappa}(x) -G_{\eps}(u(x)\sqrt{\kappa}(x)) \right] \\ 
 &= \frac{1}{\sqrt{4\pi}\eps^{\frac{m}{2}+1}} \int_{\mathcal{M}} \exp\left(-\frac{|x-\tx|^2}{4\eps}\right)\sqrt{\kappa(x)\kappa(\tx)} [u(x)-u(\tx)] dV(\tx),
\end{align*}
which satisfies
$$\L^\kappa_{\eps} u(x) = \L^\kappa u(x) + \OO (\eps), \quad\quad x\in \M.$$
We will make rigorous this formal derivation in Section \ref{sec:theory}.

We next consider the following analogue to equation \eqref{eq:pde}, defined by replacing the differential operator $\L^\kappa$ with the kernel approximation $\L^\kappa_{\eps}:$ 
\begin{align}
\L^\kappa_{\epsilon} u_{\epsilon} = f, \quad \quad x \in \M. \label{eq:pdeA}
\end{align}
Lemma \ref{lem:stability} below guarantees the existence of a unique weak solution $u_\eps \in L_0^2$ to equation \eqref{eq:pdeA} provided that $f\in L^2$ and that the original PDE is uniformly elliptic. In other words, the solution $u_{\eps}$ satisfies
\begin{align}
	\int_{\mathcal{M}} \L^{\kappa}_{\eps} u_{\eps} v = \int_{\mathcal{M}} f v , \quad \forall v\in L_0^2. \label{eq:weakform}
\end{align}
 We define $\mathcal{F}_{\epsilon}$ as the map that associates  $\theta=\log(\kappa)$ to the solution $u_\eps$ to \eqref{eq:pdeA}. Denoting $\mathcal{G}_{\epsilon} = \mathcal{D} \circ \mathcal{F}_{\epsilon}$, the approximate posterior $\pi^y_{\epsilon}$ has the following form
\begin{align}
\frac{d\pi^y_{\epsilon}}{d\pi}(\theta) \propto \exp\left( -\frac{1}{2} |y-\mathcal{G}_{\epsilon}(\theta)|^2_{\Gamma} \right).  \label{eq:posteriorA}
\end{align}
In Section \ref{sec:theory} we will establish a bound on the total variation distance between the posterior distribution $\pi^y$ defined in equation \eqref{eq:posterior} and its approximation $\pi_\eps^y.$ We note, however, that the sample-based discretization of the kernel operator $\L^\kappa_\eps$ ---that we will introduce in the next subsection--- will involve another layer of approximation not accounted for by the theory in Section \ref{sec:theory}, but necessary in practice.
\begin{remark} \label{rmk:strongsol}
	As will be seen in Section \ref{sec:theory}, a weak solution to equation \eqref{eq:pdeA} is sufficient for all the results to hold. We remark that one can show, using Fredholm alternative, the existence of a unique mean zero strong solution with the additional condition that $f$ has mean zero. 
\end{remark}

\subsection{Kernel-Based Elliptic Inverse Problem on a Point Cloud}\label{ssec:pointcloudellipticip}
In this subsection we assume that we are given a point cloud $X = \{x_1, \ldots, x_n\},$ sampled independently according to an unknown density $q$ on $\M,$ but that $\M$ is otherwise unknown. In applications, $x_i$ may represent landmarks on the underlying manifold, that may correspond to sensor locations. We consider the inverse problem of determining the value of the unknown input parameter $\kappa$ at the points $x_i \in \M$ given the observed data $y.$ Again we will follow a Bayesian perspective, defining a suitable prior $\pi_n$ over functions on the point cloud, as well as a sample-based approximation to the composition map $\G_{\eps} = \mathcal{D} \circ \F_{\eps}.$ We discuss the priors in Subsection \ref{ssec:graphprior} and the approximation to $\mathcal{G}_\eps$ in Subsection \ref{ssec:graphposterior}.

\subsubsection{Prior on Point Cloud Functions} \label{ssec:graphprior}
We now present the choice of priors that we will use for our numerical experiments in Section \ref{sec:numerics}. 
These will be defined in analogy to \eqref{eq:ctmprior}, replacing $\Delta_{\M}$ by a graph Laplacian.
More explicitly, given $n$ points $x_1,\ldots,x_n$, we set the prior to be
\begin{align}
\pi_n =\mathcal{N}(0, C_{\tau,s}^n), \quad C_{\tau,s}^n = c_n(\tau)(\tau I +\Delta_n)^{-s}, \label{eq:graphprior}
\end{align}
where $\Delta_n \in \R^{n\times n}$ is a graph Laplacian constructed with $x_1,\ldots,x_n$ and $c_n(\tau)$ is a normalizing constant. We refer to \cite{von2007tutorial} for a detailed account of graph Laplacians. Note that draws from $\pi_n$ are functions defined intrinsically in the point cloud $\M_n$ rather than on the (unknown) manifold $\M.$  The two paremeters $\tau$ and $s$ play the same role as discussed above in equation \eqref{eq:sample}.  Again samples from $\pi_n$ can be expressed by Karhunen-Lo\'eve expansion, 
\begin{align*}
	v_n=c_n(\tau)^{1/2}\sum_{i=1}^n (\tau +\lambda_i^{(n)})^{-s} \xi_i \phi_i^{(n)},
\end{align*}
where $\{(\lambda_i^{(n)},\phi_i^{(n)})\}_{i=1}^n$ are the eigenvalue-eigenvector pairs for $\Delta_n$ and $\xi_i \overset{\text{i.i.d.}}{\sim} \mathcal{N}(0,1)$. Similarly as in equation \eqref{eq:normalize}, we normalize the draws so that the variance per node is 1: 
\begin{align*}
	c_n(\tau) = \frac{n}{\sum_{i=1}^n (\tau + \lambda_i^{(n)})^{-s}}.
\end{align*}

For practical considerations, we advocate to set $\Delta_n$ as the self-tuning graph Laplacian proposed in \cite{zelnik2005self}.
To illustrate the idea, let $X=\{x_1,\ldots,x_n\}$ be the given point cloud. Then the symmetric graph Laplacian is constructed as the matrix 
\begin{align}
\Delta_n = I-A^{-1/2}S A^{-1/2}, \label{graphLaplacian}
\end{align}
where $S\in \R^{n\times n}$ is a similarity matrix and $A$ is a diagonal matrix with entries $A_{ii}=\sum_{j=1}^n S_{ij}$. We set the entries of the similarity matrix $S$ to be 
\begin{align*}
 S_{ij} =\exp\left(-\frac{|x_i-x_j|^2}{2d(i)d(j) }\right), 
\end{align*}
where $d(i)$ is the distance from $x_i$ to its $k$-th nearest neighbor, and $k$ is a tunable parameter. The idea is similar to the standard Gaussian similarities except that the local bandwidth parameter is allowed to change adaptively based on the density of points $x_i$'s.  Moreover, the bandwidth parameter is specified through $k$, a positive integer, which can be easily tuned empirically.

\subsubsection{Posterior on Point Cloud Functions} \label{ssec:graphposterior}
In this subsection we discuss how to discretize the posterior by constructing a point cloud approximation of $\mathcal{G}_{\eps}$. We first approximate $\L^{\kappa}_{\eps}$ by discretizing the integral
\begin{align*}
\mathcal{I}u(x):=\int_{\mathcal{M}} \exp \left( -\frac{|x-\tilde{x}|^2}{4\epsilon}\right)\sqrt{\kappa(\tilde{x})} u(\tilde{x})dV(\tilde{x}) 
\end{align*} 
by a Monte-Carlo sum with a reweighting by an approximate density. 
Precisely, we have 
\begin{align}\label{eq:kernelestimation}
	\mathcal{I}u(x_i) \approx \frac{1}{n} \sum_{j=1}^n \exp\left(-\frac{|x_i-x_j|^2}{4\eps}\right) \sqrt{\kappa(\smash{x_j})} u(x_j) q_{\eps}(x_j)^{-1},
\end{align}
where the approximate density, applying \eqref{eq:geps}, is given by 
\begin{align*}
q_{\eps}(x_j) = \frac{1}{\sqrt{4\pi}n\eps^{\frac{m}{2}}} \sum_{k=1}^n \exp \left(  -\frac{|x_j-x_k|^2}{4\epsilon} \right).  
\end{align*}
The  approximation in  \eqref{eq:kernelestimation} can be interpreted as combining a kernel density estimation \cite{wasserman2006all} with importance sampling \cite{agapiou2017importance,sanz2018importance}. In Section \ref{sec:numerics} we will use this observation, where the point clouds come from uniform grids. Then $\L^{\kappa}_{\eps}u$ evaluated at the point cloud is approximated by 
\begin{align}
\L^{\kappa}_{\eps}u(x_i) \approx \frac{1}{\sqrt{4\pi}n\eps^{\frac{m}{2}+1}} \sum_{j=1}^n \exp\left(-\frac{|x_i-x_j|^2}{4\epsilon}\right) \sqrt{\kappa(x_i)\kappa(\smash{x_j})} q_{\eps}(x_j)^{-1}[u(x_i)-u(x_j)]:=L^{\kappa}_{\eps,n} u(x_i).  \label{eq:repL}
\end{align}
More concisely, we can write $L^{\kappa}_{\eps,n}$ in matrix form in a series of steps.
Define $H$ to be the kernel matrix with entries $H_{ij}=\exp\left(-|x_i-x_j|^2/4\eps\right)$. Let $Q$ be a vector with entries $Q_i=\sum_{j=1}^n H_{ij}$ and define $W$ to be the matrix with entries $W_{ij}=H_{ij}\sqrt{\kappa(x_i)\kappa(x_j)}Q_j^{-1}$. Then we have 
\begin{align}
	L^{\kappa}_{\eps,n}= \frac{1}{\eps}(D-W), \label{eq:discreteL}
\end{align} 
where $D$ is a diagonal matrix with entry $D_{ii}= \sum_{j=1}^n W_{ij}$. Notice that the above construction resembles that of the unnormalized graph Laplacian. Indeed, if $\kappa\equiv 1$, then \eqref{eq:discreteL} is  exactly the unnormalized graph Laplacian up to a factor of the density \cite{coifman2006diffusion}. 

Given the above discretization, we consider the following analogue to equation \eqref{eq:pdeA}, by replacing $\L^{\kappa}_{\eps}$ with $L_{\eps,n}^{\kappa}$ and restricting $f$ to the point cloud:
\begin{align}
L_{\eps,n}^{\kappa} u_n = f_n, \label{eq:pdeM} 
\end{align}
where $f_n$ is the $n$-dimensional vector with entries $f(x_i)$, or an approximation thereof when $f$ is not smooth. One can see from  \eqref{eq:repL} that $L^{\kappa}_{\eps,n}$ is self-adjoint and positive semi-definite under the weighted inner product $\langle u,v \rangle_q:=\frac{1}{n}\sum_{j=1}^n u(x_i) v(x_i) q_{\eps}(x_i)^{-1}$. Hence $L^{\kappa}_{\eps,n}$ admits a nonnegative spectrum $\{\lambda_i\}_{i=1}^n$ with $\lambda_1=0$ and an orthonormal basis of eigenfunctions $\{v_i\}_{i=1}^n$ wih respect to $\langle\cdot,\cdot\rangle_q$, with $v_1\equiv 1$. We then set the solution to be 
\begin{align}
u_n=\sum_{i=2}^n \frac{f_n^i}{\lambda_i} v_i, \label{eq:graphsol}
\end{align}
where the $f_n=\sum_{i=1}^n f_n^i v_i$. Notice that the mean zero condition of $u$ translates into $\langle u, 1\rangle_q=0$, taking into account the density. By the orthogonality of the $v_i$'s, we see that the solution $u_n$ in \eqref{eq:graphsol} satisfies $\langle u_n, 1\rangle_q=0$ and moreover, $\{v_2,\ldots,v_n\}$ forms a basis for $\ell_0^2=\{v: \langle v, 1\rangle_q=0\}$, which is the discrete analogue of $L_0^2$.  One can also check that $u_n$ satisifies $\langle L^{\kappa}_{\eps,n} u_n, v \rangle_q = \langle f, v \rangle_q$ for all $v \in \ell_0^2$, is consistent with equation \eqref{eq:weakform}. We remark that if in addition $\langle f,1\rangle_q=0$, then $u_n$ given by equation \eqref{eq:graphsol} is a strong solution of equation \eqref{eq:pdeM},  in analogy to Remark \ref{rmk:strongsol}.

Hence we can now define the discrete forward map $F_{\epsilon,n}:\mathbb{R}^n \mapsto \mathbb{R}^n$ as the map that associates $\theta_n=\log(\kappa_n):=\big(\log(\kappa(x_1)),\ldots,\log(\kappa(x_n))\big)$ to the solution $u_n$. Approximating the pointwise observation map is straightforward. We may also approximate the smoothed observation map introduced in Subsection \ref{sec:obsmap} by  Monte-Carlo as follows:
\begin{align*}
 \ell^{(n)}_j(u_n)=\frac{1}{n} \sum_{k=1}^n K(x_j,x_k)u_n(x_k) q_{\eps}(x_k)^{-1}. 
\end{align*} 
In either case, denoting $D_n(u_n)=\big(\ell_1^{(n)}(u_n),\ldots,\ell^{(n)}_J(u_n)\big)^T$ and $G_{\eps,n}= D_{n} \circ F_{\eps,n}$, the graph posterior has the following form
\begin{align}
\frac{d\pi^y_{\eps,n}}{d\pi_n}(\theta_n) \propto \exp\left( -\frac{1}{2} |y-G_{\eps,n}(\theta_n)|^2_{\Gamma} \right). \label{eq:discretepostdensity}
\end{align}
A full analysis of the convergence of the sample-based posteriors $\pi^y_{\eps,n}$ to the ground-truth posterior $\pi^y$ is beyond the scope of this paper. For a \emph{ linear} regression problem, the convergence of such graph-based posteriors has been established in \cite{trillos2017consistency} and \cite{garcia2018continuum} using spectral graph theory and variational techniques.

\section{Analysis of Kernel Approximation to the Forward and Inverse Problem}\label{sec:theory}
In this section we study the error incurred by replacing the differential operator in the forward map by its kernel approximation, and the effect of such error in the posterior solution to the Bayesian inverse problem. The approximation of the forward map is analyzed in Subsection \ref{ssec:forwardanalysis} and the approximation of the posterior in Subsection \ref{ssec:inverseanalysis}.

\subsection{Forward Map Approximation}\label{ssec:forwardanalysis}
The main result of this subsection is the following theorem which bounds the difference between the solution to the PDE \eqref{eq:pde} and the solution to the kernel-based equation \eqref{eq:pdeA}.
\begin{theorem}[Forward map approximation]  \label{thm:1}
Suppose that   $f \in \mathcal{C}^{3,\alpha}$ and $\kappa \in \mathcal{C}^4$. 
Let $u$ solve $\L^\kappa u = f$, and $u_{\eps}$ solve $\L^\kappa_\eps u_{\eps}= f$ weakly in $L_0^2$. 
Then for $\frac{1}{4}<\beta <\frac{1}{2}$ and $\eps$ small enough depending on $\beta$, 
\begin{align*}
    \|u-u_{\eps}\|_{L^2}  \leq CA(\kappa)\|f\|_{H^3} \eps^{4\beta -1},
\end{align*}
where $C$ is a constant depending only on $\M$ and 
\begin{align*}
	A(\kappa)= \left(1\vee\sqrt{ \kmin^{-5} + \kmin^{-6}\big(\|\kappa\|^2_{\mathcal{C}^3}+\|\kappa\|_{\mathcal{C}^3} \big) 
+ \kmin^{-7}\big(\|\kappa\|^2_{\mathcal{C}^3}+\|\kappa\|_{\mathcal{C}^3} \big)^2+\kmin^{-8}\big(\|\kappa\|^2_{\mathcal{C}^3}+\|\kappa\|_{\mathcal{C}^3} \big)^3}\right) \|\sqrt{\kappa}\|_{\mathcal{C}^4}.
\end{align*}
\end{theorem}
The novelty is to generalize previous analysis \cite{coifman2006diffusion,gilani2019approximating} to the case of anisotropic diffusions and, more importantly, to keep track of the dependence $A(\kappa)$ of the error bound on the diffusion coefficient $\kappa.$ As we will show in Subsection \ref{ssec:inverseanalysis}, understanding this dependence is a crucial ingredient in establishing an approximation result for the inverse problem.

The proof of Theorem \ref{thm:1} follows the classical numerical analysis argument of combining stability and consistency, coupled with an $H^4$ norm estimate for solutions to PDE \eqref{eq:pde}.  Lemma \ref{lem:stability} below establishes the stability of solutions to the kernel-based equation \eqref{eq:pdeA}, Lemma \ref{lem:operator} shows consistency, and Lemma \ref{lem:H4norm} shows an $H^4$ norm  bound on solutions to \eqref{eq:pde}. The proof of Theorem \ref{thm:1} will be given at the end of this subsection by combining these three lemmas. To streamline the presentation we postpone the proofs of the lemmas to an Appendix.

\begin{lemma}[Stability] \label{lem:stability}
The equation $\L_\eps^\kappa u_\eps=f,$ with $f \in L^2$ and $\kappa$ satisfying $ \kappa(x) \ge \kmin$ for a.e. $x\in \M$ has a unique weak solution $u_\eps \in L_0^2.$ Moreover, there is $C>0$ independent of $\eps$ and $\kappa$ such that 
\begin{equation}\label{eq:boundueps}
    \|u_\eps\|_{L^2} \leq C \kmin^{-1} \|f\|_{L^2}. 
\end{equation}
\end{lemma}

The next lemma makes rigorous the argument in Subsection \ref{ssec:kernelellipticip}  and characterizes the error between $\L^{\kappa}$ and $\L^{\kappa}_{\eps}$ by  accounting for its dependence on $\kappa$. 

\begin{lemma}[Consistency] \label{lem:operator} 
Let  $u \in \mathcal{C}^4$ and $\kappa\in \mathcal{C}^4$. Then, for $\frac{1}{4}<\beta<\frac{1}{2}$ and $\eps$ sufficiently small depending on $\beta$, we have  
\begin{align*}
   \|(\L^\kappa_{\eps}-\L^\kappa)u\|_{L^2}\leq C (1 \vee \|u\|_{H^4})\|\sqrt{\kappa}\|_{\mathcal{C}^4}\eps^{4\beta-1}. 
\end{align*}
\end{lemma}
\begin{remark}
In the proof of Lemma \ref{lem:operator}, found in the Appendix, we cannot set $\beta=\frac{1}{2}$. However we can choose $\beta$ arbitrarily close to $\frac{1}{2}$ so that the rate is essentially $O(\eps)$. We remark that the proof of Lemma \ref{lem:operator} suggests that the $\mathcal{C}^3$ assumption in \cite{coifman2006diffusion} may not be sufficient.
\end{remark}

The last lemma bounds the $H^4$ norm of the solution to equation \eqref{eq:pde} in terms of the diffusion coefficient $\kappa.$
\begin{lemma}[$H^4$-norm bound]  \label{lem:H4norm}
Suppose that $\kappa \in C^4$ and $f \in C^{3,\alpha}$ with $0<\alpha<1.$
Let $u\in \mathcal{C}^5$ be the zero-mean solution to the equation $\L^\kappa u=f$.  Then 
\begin{align*}
    \|u\|_{H^4}^2\leq C\|f\|^2_{H^3}\Big[  \kmin^{-5} + \kmin^{-6}\big(\|\kappa\|^2_{\mathcal{C}^3}+\|\kappa\|_{\mathcal{C}^3} \big) 
+ \kmin^{-7}\big(\|\kappa\|^2_{\mathcal{C}^3}+\|\kappa\|_{\mathcal{C}^3} \big)^2+\kmin^{-8}\big(\|\kappa\|^2_{\mathcal{C}^3}+\|\kappa\|_{\mathcal{C}^3} \big)^3\Big],
\end{align*}
where $C$ is a constant that depends only on $\mathcal{M}$.
\end{lemma}

We are now ready to prove Theorem \ref{thm:1}.

\begin{proof}[Proof of Theorem \ref{thm:1}]
Recall that we need to show that
\begin{align*}
    \|u-u_{\eps}\|_{L^2}  \leq CA(\kappa)\|f\|_{H^3} \eps^{4\beta -1},
\end{align*}
where $u$ solves $\L^\kappa u = f$ and $u_{\eps}$ solves $\L^\kappa u_\eps = f$ weakly, and  $A(\kappa)$ is defined in the statement of Theorem \ref{thm:1}.
Notice that in the weak sense
\begin{align*}
	\L^\kappa_{\eps} (u-u_{\eps}) =\L^\kappa_{\eps}u -f = \L^\kappa_{\eps}u - \L^\kappa u. 
\end{align*}
Hence using Lemma \ref{lem:stability} for the first inequality, and Lemma \ref{lem:operator} for the second one noting that  $f \in \mathcal{C}^{3,\alpha}$ implies that $u \in \mathcal{C}^5$ \cite{gilbarg2015elliptic}, we have 
\begin{align*}
\|u-u_\eps\|_{L^2} &\leq C \kmin^{-1} \|(\L^\kappa_\eps - \L^\kappa)u\|_{L^2} \\ 
& \le C \kmin^{-1} \big(1 \vee \|u\|_{H^4}\big) \|\sqrt{\kappa}\|_{\mathcal{C}^4} \eps^{4\beta-1}. 
\end{align*} 
The result follows by combining this inequality with the bound on $\|u\|_{H^4}$ derived in Lemma \ref{lem:H4norm}. 
\end{proof}

\subsection{Posterior Approximation}\label{ssec:inverseanalysis}
In this subsection we characterize the total variation distance between the two posteriors:
\begin{align*}
    \frac{d\pi^{y}}{d\pi}(\theta) &=\frac{1}{Z} \exp\left(-\frac{1}{2} |y-\mathcal{G}(\theta)|^2_{\Gamma}\right), \quad \quad \quad\,\,\,\,\,\,  Z:=\int \exp\left(-\frac{1}{2} |y-\mathcal{G}(\theta)|^2_{\Gamma}\right) d\pi(\theta),  \\
\frac{d\pi^{y}_{\eps}}{d\pi} (\theta)&= \frac{1}{Z_{\eps}}\exp\left( -\frac{1}{2} |y-\mathcal{G}_{\eps}(\theta)|_{\Gamma}^2 \right), \quad \quad \quad Z_{\eps}:=\int \exp\left( -\frac{1}{2} |y-\mathcal{G}_{\eps}(\theta)|_{\Gamma}^2 \right) d\pi(\theta),
\end{align*}
where $Z$ and $Z_\eps$ are normalizing constants and recall that $\mathcal{G}(\theta)=(\ell_1(u),\ldots,\ell_J(u))^T$ and $\mathcal{G}_{\eps}(\theta)=(\ell_1(u_{\eps}),\ldots,\ell_J(u_{\eps}))^T$ where $\ell_j$'s are bounded linear functionals on $L^2$. 

The main result is Theorem \ref{thm:2} below. Its proof relies on Theorem \ref{thm:1} and a standard argument \cite{AS10,trillos2016bayesian,sanzstuarttaeb} for the analysis of approximations of Bayesian inverse problems. In particular, the proof makes use of the integrability of the function $A(\kappa)$ defined in Theorem \ref{thm:1} with respect to the prior $\pi$, guaranteed by Fernique's theorem \cite{fernique1970integrabilite}. 

\begin{theorem}[Posterior approximation] \label{thm:2}
Let $\pi$ be a Gaussian measure on $\mathcal{C}^4$, and suppose that $f \in C^{3,\alpha}$ for $0<\alpha<1.$ Then for any $\frac{1}{4}<\beta<\frac{1}{2}$ and $\eps$ sufficiently small depending on $\beta$, 
\begin{align*}
    \dtv(\pi^y,\pi^y_{\eps})\leq C \eps^{4\beta-1},
\end{align*}
where $C$ is constant depending only on $\mathcal{M}$.
\end{theorem}
\begin{proof}
We have 
\begin{align*}
d_{\text{TV}}(\pi^{y},\pi^{y}_{\eps})&  =\int  \left| \frac{1}{Z_{\eps}} \exp\left( -\frac{1}{2} |y-\mathcal{G}_{\eps}(\theta)|_{\Gamma}^2 \right)-\frac{1}{Z} \exp\left(-\frac{1}{2} |y-\mathcal{G}(\theta)|^2_{\Gamma}\right)
 \right| d\pi(\theta) \\
& \leq \int \left|\frac{1}{Z_{\eps}}- \frac{1}{Z}   \right| \exp\left( -\frac{1}{2} |y-\mathcal{G}_{\eps}(\theta)|_{\Gamma}^2 \right) d\pi(\theta)\\
& + \int \frac{1}{Z} \left|\exp\left( -\frac{1}{2} |y-\mathcal{G}_{\eps}(\theta)|_{\Gamma}^2 \right)- \exp\left(-\frac{1}{2} |y-\mathcal{G}(\theta)|^2_{\Gamma}\right) \right| d\pi(\theta)\\
&\leq \left| \frac{1}{Z_{\eps}}-\frac{1}{Z} \right| +\int \frac{1}{Z} \left|\exp\left( -\frac{1}{2} |y-\mathcal{G}_{\eps}(\theta)|_{\Gamma}^2 \right)- \exp\left(-\frac{1}{2} |y-\mathcal{G}(\theta)|^2_{\Gamma}\right) \right| d\pi(\theta)\\
&= \frac{|Z-Z_{\eps}|}{ZZ_{\eps}}+\int \frac{1}{Z} \Big|\exp\left( -\frac{1}{2} |y-\mathcal{G}_{\eps}(\theta)|_{\Gamma}^2 \right)- \exp\left(-\frac{1}{2} |y-\mathcal{G}(\theta)|^2_{\Gamma}\right) \Big| d\pi(\theta).
\end{align*}
Since $Z>0$, by Lemma \ref{lem:Z} below, $Z_{\eps}>\frac{Z}{2}$ for $\eps$ small enough. Hence we have 
\begin{align*}
	d_{\text{TV}}(\pi^y,\pi^y_{\eps}) \leq C \int \Big|\exp\left( -\frac{1}{2} |y-\mathcal{G}_{\eps}(\theta)|_{\Gamma}^2 \right)- \exp\left(-\frac{1}{2} |y-\mathcal{G}(\theta)|^2_{\Gamma}\right) \Big| d\pi(\theta).
\end{align*}
Then using Lemma \ref{lem:Z} again it follows that 
\begin{align*}
	d_{\text{TV}}(\pi^y,\pi^y_{\eps}) \leq C \eps^{4\beta-1},
\end{align*}
where $C$ is independent of $\eps$. 
\end{proof}

The following lemma was used in the proof of Theorem \ref{thm:2}. The proof can be found in the Appendix and makes use of the integrability of $A(\kappa)$ with respect to the prior.
\begin{lemma} \label{lem:Z}
For $\frac{1}{4}<\beta<\frac{1}{2}$ and $\epsilon$ small enough depending on $\beta$, we have 
\begin{align*}
	\int \left|\exp\left( -\frac{1}{2} |y-\mathcal{G}_{\eps}(\theta)|_{\Gamma}^2 \right)- \exp\left(-\frac{1}{2} |y-\mathcal{G}(\theta)|^2_{\Gamma}\right) \right| d\pi(\theta)\leq C \eps^{4\beta-1},
\end{align*}
where $C$ is independent of $\eps$. 
\end{lemma}

\begin{remark}
Similarly as for Lemma \ref{lem:operator} our proof fails when $\beta=\frac{1}{2}$. However one can choose $\beta$ arbitrarily close to $\frac{1}{2}$ so that the rate in \ref{thm:2} is essentially $O(\eps)$.
\end{remark}

\section{Numerical Experiments}\label{sec:numerics}

In this section we investigate numerically the point cloud formulation of the inverse problem introduced in Subsection \ref{ssec:pointcloudellipticip}. We start in Subsection \ref{ssec:implementation} by considering some aspects of the implementation. Then in Subsections \ref{ssec:ex1}, \ref{ssec:ex2}, and \ref{ssec:ex3} we give three numerical examples, where the underlying manifold is chosen to be an ellipse, the torus, and a cow-shaped manifold. In Subsection \ref{sec:HB}, we study a hierarchical approach where the prior length-scale parameter is learned from data.

\subsection{Implementation} \label{ssec:implementation}
When it comes to practical applications, care must be taken when one chooses the parameters. Central in our kernel method is the parameter $\eps$.  While Theorem \ref{thm:1} characterizes the error in approximating $\L^{\kappa}$ with $\L^{\kappa}_{\eps}$ and suggests the consistency of the estimator as $\eps$ goes to 0, in practice one cannot take $\epsilon$ too small as we explain now. One can indeed establish the consistency of $L^{\kappa}_{\eps,n}$ with $\L^{\kappa}_{\eps}$, using the same discrete estimation technique as in \cite{berry2016variable,bs:16,gilani2019approximating}. We should point out that while the resulting discrete error bound in \cite{berry2016variable,bs:16,gilani2019approximating} does not show any dependence on $\kappa$, which is needed for proving the convergence of the discrete posterior density estimate in \eqref{eq:discretepostdensity} to \eqref{eq:posteriorA}, this result is sufficient for understanding the consistency of $L^{\kappa}_{\eps,n}$ with $\L^{\kappa}_{\eps}$. Specifically, for point cloud with distribution characterized by density $q(x)$, defined with respect to the volume form inherited by $\mathcal{M}\subset\mathbb{R}^d$, the discrete estimate for fixed-bandwidth Gaussian kernel (e.g., see Corollary~1 of \cite{berry2016variable} with $\alpha=1/2, \beta=0$ in their setup) states that the sampling error for obtaining an order-$\epsilon^2$ of the density $q$ with $q_\epsilon$ is of order $\mathcal{O}(q(x_i)^{1/2}n^{-1/2}\epsilon^{-(2+m/4)})$ and the error between $L^{\kappa}_{\eps,n}$ and $\L^{\kappa}_{\eps}$ is of order $n^{-1/2}\epsilon^{-(1/2+m/4)}$. The fact that $\epsilon$ appears in the denominator of these estimates suggests that one cannot take $\epsilon$ too small in practice and it also implies that $\epsilon$ should be adequately scaled with the size of the data, $n$. Since a direct use of these estimates requires knowing the constants that depend on the volume of $\mathcal{M}$ that are difficult to estimate, we instead adopt an automated empirical method for choosing $\epsilon.$ Precisely, we follow \cite{coifman2008graph} and plot 
$$T(\eps):=\sum_{i,j} \exp\left(-\frac{|x_i-x_j|^2}{4\eps}\right)$$
as a function of $\eps$ and choose $\eps$ to be in the region where $\log\big(T(\eps)\bigr)$ is approximately linear.

In the following three subsections we demonstrate the local kernel method for solving inverse problems through three numerical examples. In the first two examples, the embeddings are known and we set the model analytically, i.e., we first choose the ground truth $\kappa^{\dagger}$ and $u^{\dagger}$, and then compute the corresponding $f$ as 
\begin{align}\label{eq:RHS}
f = \text{div} (e^{\theta} \nabla u) = \frac{1}{\sqrt{\text{det}g}} \partial_i \left( e^{\theta}  g^{ij} \sqrt{\text{det} g}\partial_j u\right) ,
\end{align}
where $g$ is the Riemannian metric on $\mathcal{M}$. The third example will be an artificial surface where the embedding is unknown. We will then generate the truth using our kernel method. 
We will use the pCN algorithm to sample from the posterior \cite{cotter2013,bertozzi2018uncertainty,trillos2017consistency}. This is a Metropolis-Hastings algorithm with proposal mechanism to move from $u_n$ to $u_n^*$ given by 
\begin{equation}\label{eq:pcnkernel}
u_n^\star \sim \left( 1 - \beta^{2} \right)^{1/2}u_n + \beta \xi^{(\sam)}, 
\quad \xi\sim \pi_n =\mathcal{N}(0, C_{\tau,s}^n),
\end{equation}
where $\beta \in (0,1)$ is a tuning parameter. Note that if $u_n\sim \pi_n$ then $u_n^* \sim \pi_n$ showing that the prior is invariant for this kernel. Moreover, it is not difficult to see that detailed balance holds, and as a consequence the Metropolis-Hastings accept/reject mechanism involves only evaluation of the likelihood function. The advantage of pCN in our setting over a standard random walk or Langevin algorithm is that the rate of convergence of pCN does not deteriorate with $n;$ this has been established rigorously for a linear inverse problem in \cite{trillos2017consistency}.

\begin{remark}\label{rmk:compcost}
At each iteration of the MCMC algorithm, the forward map involves an eigenvalue decomposition of a different matrix for different $\theta$'s as shown in Subsection \ref{ssec:graphposterior}. Hence large $n$'s are not favored for computational purposes and this can be an issue for high dimensional $\M$'s where the number of points grow as $n^m$ if one discretizes each dimension by $n$.  
\end{remark}

\subsection{One-Dimensional Elliptic Problem on an Unknown Ellipse}\label{ssec:ex1}
In this subsection we take $\mathcal{M}$ to be an ellipse with semi-major and semi-minor axis of length $a=3$ and 1, embedded through 
\begin{align}
\iota (\omega) = (\cos\omega,a\sin\omega)^T, \quad \omega \in [0,2\pi], \label{eq:ellipse}
\end{align}
and the Riemannian metric is 
\begin{align*}
g_{11}(\omega) = \sin^2\omega + a^2 \cos^2\omega. 
\end{align*}
The truth is set to be 
\begin{align*}
\kappa^{\dagger}& = 2 +\cos\omega, \quad u^{\dagger} =\cos\omega,  
\end{align*}
and right-hand side $f$ in equation \eqref{eq:pde} is defined through equation \eqref{eq:RHS}. One can check that both $u^{\dagger}$ and $f$ have mean zero, i.e., $\int_0^{2\pi} u^{\dagger}(\omega) \sqrt{g_{11}(\omega)} d\omega=\int_0^{2\pi} f(\omega) \sqrt{g_{11}(\omega)}d\omega=0$. We generate the point cloud $\{x_1,\ldots,x_n\}$ according to \eqref{eq:ellipse} from a uniform grid of $\omega$ of size $n=400$. The observations are given as noisy pointwise evaluations at subsets of the point cloud:
\begin{align*}
\ell_j=u(x_j)+\eta_j, \quad \quad j=1,\ldots,J,
\end{align*} 
where  $\eta_j \sim \mathcal{N}(0,\sigma^2)$ are assumed to be independent. We will take $J=100$, 200, 400 respectively with noise size varying as $\sigma=0.01, 005, 0.1$. As discussed in Subsection \ref{ssec:graphprior}, we construct the prior with a self-tuning graph Laplacian, using $k=2$ neighbors. We empirically discover that such choice of $k$ gives the best spectral approximation towards the Laplace-Beltrami operator on the ellipse, which has spectrum $\{i^2\}_{i=0}^{\infty}$ with eigenfunctions $\{\sin(i\omega),\cos(i\omega)\}_{i=0}^{\infty}$. We also tune empirically the parameters in \eqref{eq:graphprior} as $\tau=0.05$ and $s=4$.

In Figure \ref{figure:smoothE}, we plot the posterior means as functions of $\omega\in [0,2\pi]$ and the 95\% credible intervals for different $\sigma$ and $J$'s. While the point cloud Bayesian solution is only defined at the discrete point cloud, to ease the visualization we represent the outcome as continuous functions defined on $\omega\in [0,2\pi]$.
We see that the truth is mostly captured in the Bayesian confidence intervals. To quantify the error of reconstruction, we compute the relative $\ell_2$ distances between the posterior mean $\bar{\kappa}$ and the truth $\kappa^{\dagger}$. Moreover, we compute the  solution $\bar{u}$ of \eqref{eq:pdeM} with $\bar{\kappa}$ and its relative $\ell_2$ distance to the true solution $u^{\dagger}$. As shown in Table \ref{table:smoothE},  the reconstruction error for $u^{\dagger}$ is much smaller than the relative noise level defined as $\sqrt{n}\sigma/\|u^{\dagger}\|_2$.

\begin{table}[h!]
\centering
\begin{tabular}{ |c|c|c|c|c|c|c|c|c|c| } 
 \hline
 \multicolumn{1}{|c|}{$\sigma$} & \multicolumn{3}{c|}{ 0.01} &    \multicolumn{3}{c|}{ 0.05}  &  \multicolumn{3}{c|}{ 0.1} \\
 \hline
 $J$ &100 &200 & 400 & 100 & 200 & 400 & 100 & 200 & 400\\ 
 \hline
 $\|\bar{\kappa}-\kappa^{\dagger}\|_{2}/\|\kappa^{\dagger}\|_2$ & 0.60\% & 0.80\% &0.62\% & 2.85\% & 1.96\% & 2.18\% & 5.46\% & 3.90\%  & 3.45\%\\ 
\hline
 $\|\bar{u}-u^{\dagger}\|_{2}/\|u^{\dagger}\|_2$ & 0.26\% & 0.23\% &0.23\%  & 1.08\% & 0.83\%& 0.90\% & 1.70\% & 1.37\% & 1.70\%\\ 
 \hline
  \multicolumn{1}{|c|}{$\sqrt{n}\sigma/\|u^\dagger\|_2$} & \multicolumn{3}{c|}{ 1.41\%} &    \multicolumn{3}{c|}{7.07\%}  &  \multicolumn{3}{c|}{14.14\%}\\
  \hline
\end{tabular}
\caption{Relative error of $\bar{\kappa}$ and $\bar{u}$ for different noise level, $\sigma$'s and number of observations, $J$. In the last row, the relative noise level  for each $\sigma$ is reported for diagnostic purposes. Particularly, note that the reconstruction error for $u^\dagger$ is much smaller than the relative noise level.}
\label{table:smoothE}
\end{table}


\begin{figure}[!htb]
\minipage{1\textwidth}
\minipage{0.3333\textwidth}
\subcaption{$\sigma=0.01, J=100$.}\vspace{-8pt}
  \includegraphics[width=\linewidth]{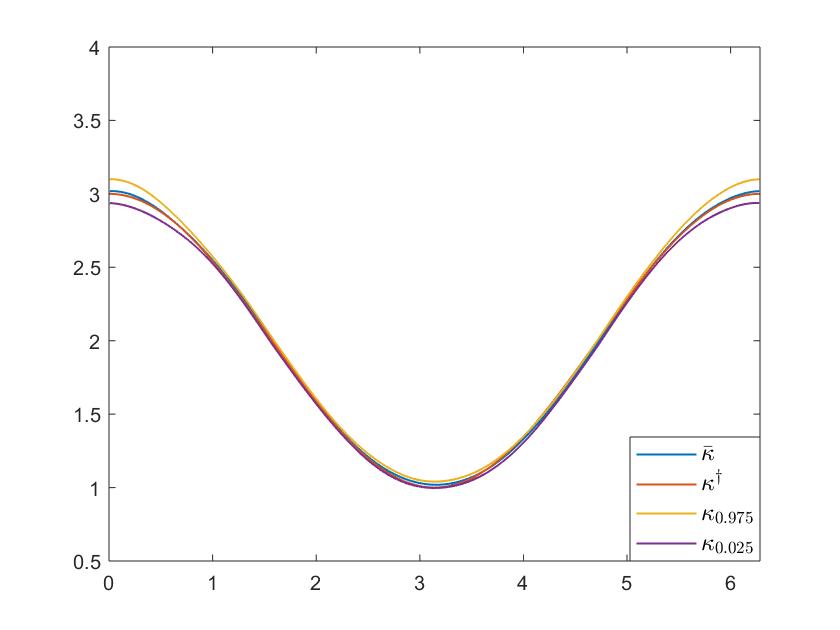}
\label{fig:awesome_image1}
\endminipage\hfill
\minipage{0.3333\textwidth}
\subcaption{$\sigma=0.01, J=200$.}\vspace{-8pt}
  \includegraphics[width=\linewidth]{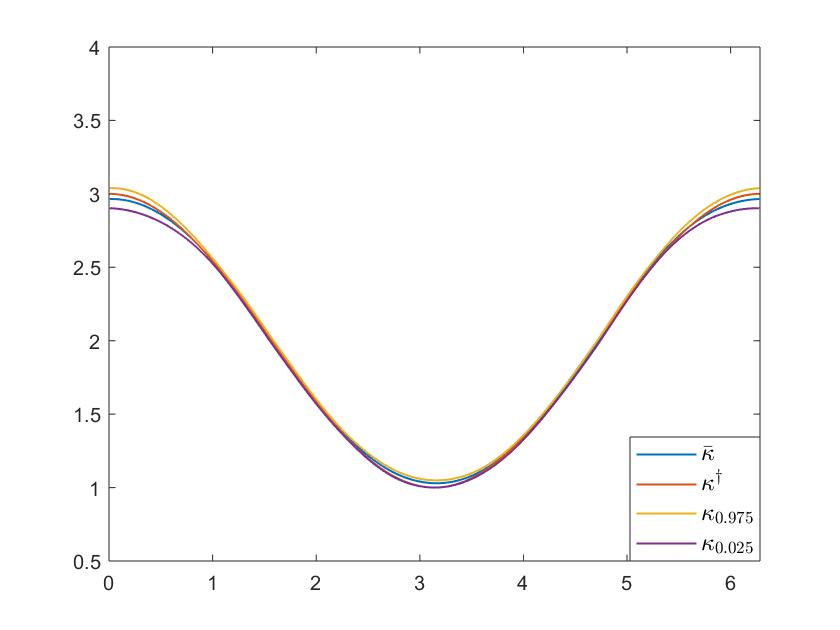}
\label{fig:awesome_image2}
\endminipage\hfill
\minipage{0.3333\textwidth}
\subcaption{$\sigma=0.01, J=400$.}\vspace{-8pt}
  \includegraphics[width=\linewidth]{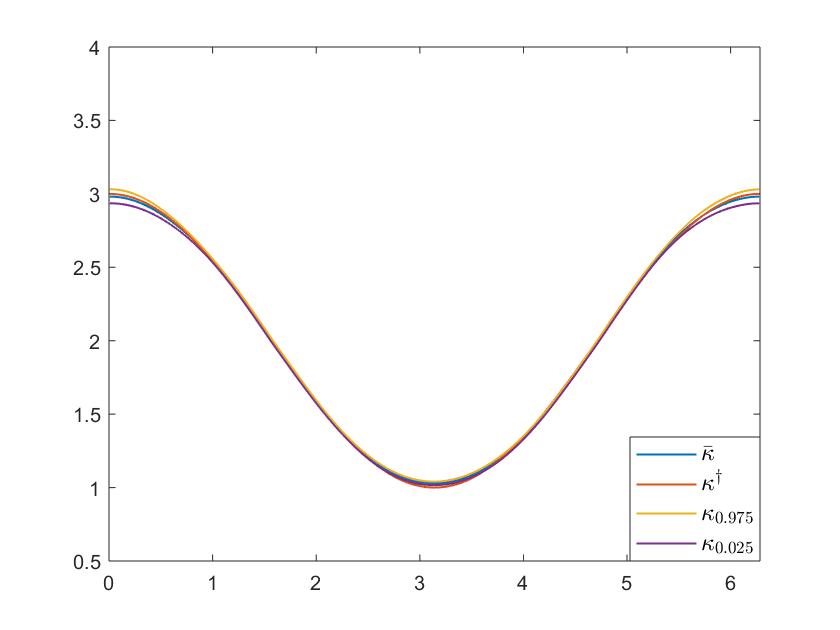}
\label{fig:awesome_image2}
\endminipage
\endminipage\hfill
\minipage{1\textwidth}
\minipage{0.3333\textwidth}
\subcaption{$\sigma=0.05, J=100$.}\vspace{-8pt}
  \includegraphics[width=\linewidth]{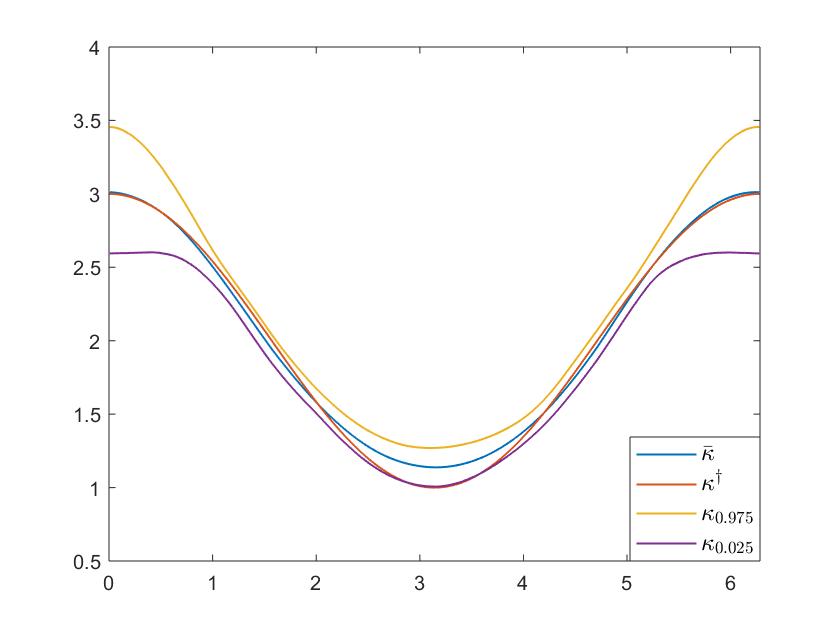}
\label{fig:awesome_image2}
\endminipage\hfill
\minipage{0.3333\textwidth}
\subcaption{$\sigma=0.05, J=200$.}\vspace{-8pt}
  \includegraphics[width=\linewidth]{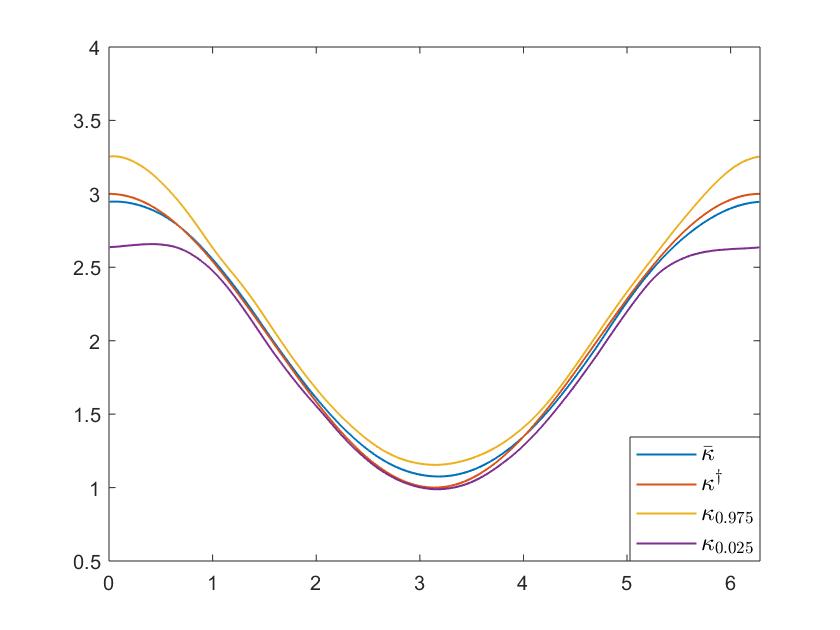}
\label{fig:awesome_image2}
\endminipage\hfill
\minipage{0.3333\textwidth}
\subcaption{$\sigma=0.05, J=400$.}\vspace{-8pt}
  \includegraphics[width=\linewidth]{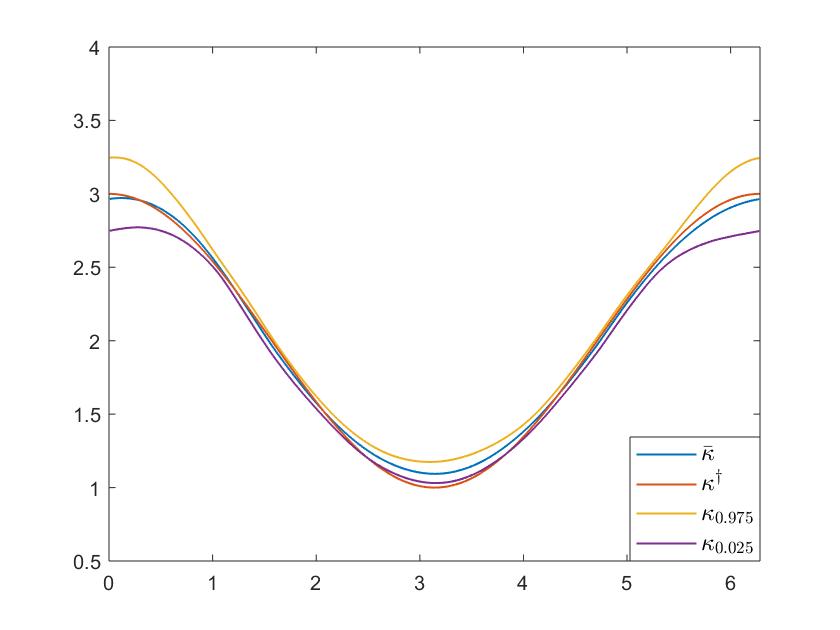}
\label{fig:awesome_image2}
\endminipage
\endminipage\hfill
\minipage{1\textwidth}
\minipage{0.3333\textwidth}
\subcaption{$\sigma=0.1, J=100$.}\vspace{-8pt}
  \includegraphics[width=\linewidth]{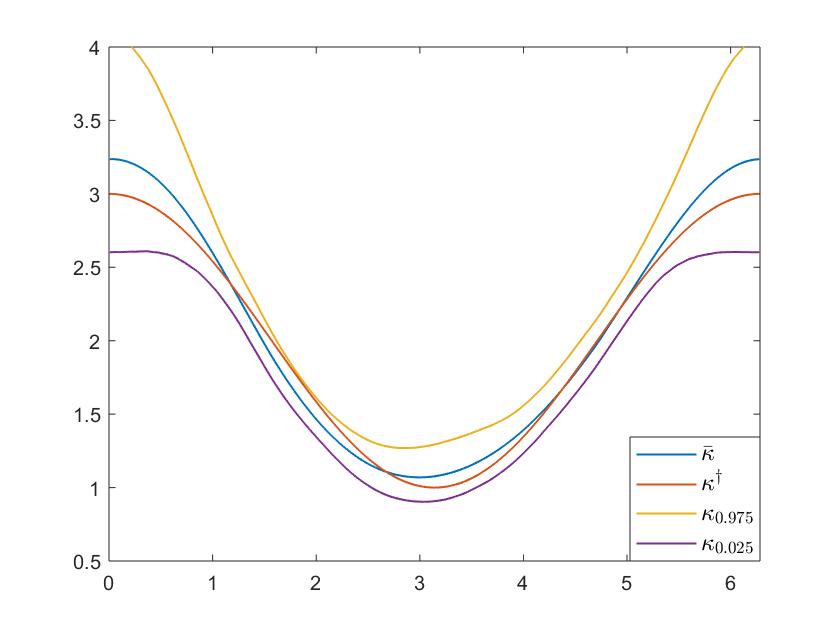}
\label{fig:awesome_image2}
\endminipage\hfill
\minipage{0.3333\textwidth}
\subcaption{$\sigma=0.1, J=200$.}\vspace{-8pt}
  \includegraphics[width=\linewidth]{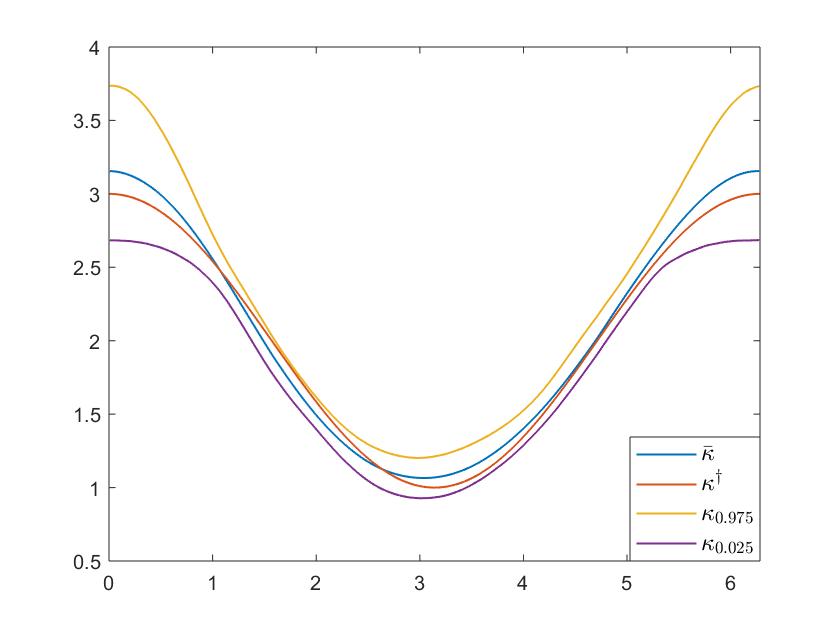}
\label{fig:awesome_image2}
\endminipage\hfill
\minipage{0.3333\textwidth}
\subcaption{$\sigma=0.1, J=400$.}\vspace{-8pt}
  \includegraphics[width=\linewidth]{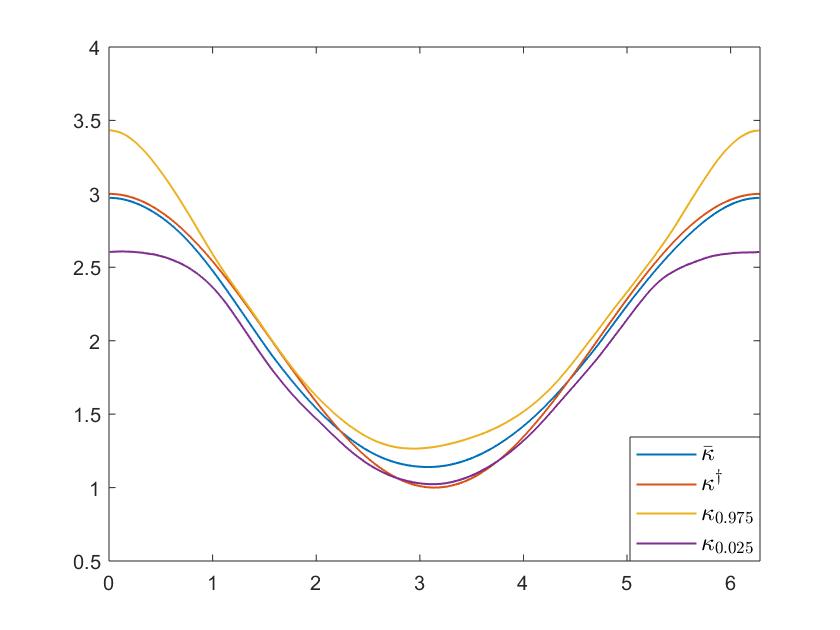}
\label{fig:awesome_image2}
\endminipage
\endminipage\hfill 
\vspace{-20pt}
\caption{Posterior means and 95\% credible intervals for different $\sigma$'s and $J$'s. Here $\kappa_{0.025}$ and $\kappa_{0.975}$ represent the 2.5\% and 97.5\% posterior quantiles respectively.}
\label{figure:smoothE}
\end{figure}


\begin{remark}\label{rmk:limitprior}
Since our prior is on $\theta = \log(\kappa)$, we are actually approximating $\log(2+\cos\omega)$ with trigonometric functions and hence the truth $\kappa^{\dagger}$ is not simply the combination of the first two eigenfunctions in the prior. In other words, although the truth $\kappa^{\dagger}$ is in the support of the prior, the fact that its coordinates in the eigenbasis do not decay like that in the expansion  \eqref{eq:sample} makes it difficult to reconstruct. 
\end{remark}
Regarding Remark \ref{rmk:limitprior}, we consider another prior with self-tuning graph constructed with $k=0.2n$ points. This new graph Laplacian gives a worse spectral approximation to the Laplace Beltrami operator in the underlying manifold, as its spectrum saturates instead of growing at the appropriate rate. In other words, the basis functions associated with the high frequencies will be given more weight in the expansion \eqref{eq:sample}. This can be beneficial in practical applications since it effectively enlarges the support of the prior. Below in Figure \ref{figure:roughE}, we solve the inverse problem using this new prior in the case $\sigma=0.1$. The parameters are tuned empirically: $\tau=0.75$, $s=8$. It can be seen that although the reconstructions are rougher than those in Figure \ref{figure:smoothE}, they capture better the shape of $\kappa^{\dagger}$, with the help of the higher frequencies. Essentially, larger $\tau$ (corresponds to more nontrivial modes in the representation in \eqref{eq:sample}) gives less bias but larger variance, which is consistent with the theory of nonparameteric statistical estimation (e.g, Section~1.7 of \cite{tsybakov2008introduction}).

\begin{figure}[!htb]
\minipage{1\textwidth}
\minipage{0.3333\textwidth}
\subcaption{$\sigma=0.1, J=100$.}\vspace{-8pt}
  \includegraphics[width=\linewidth]{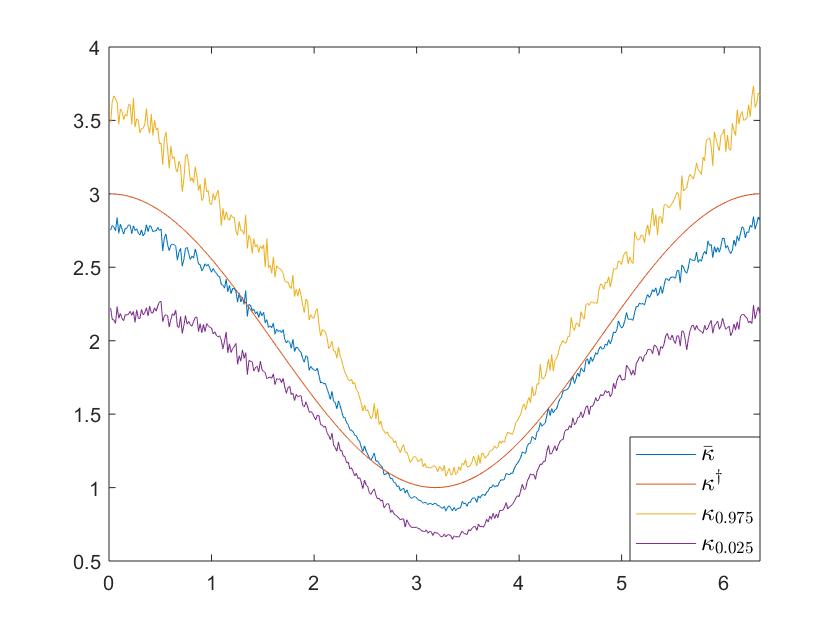}
\label{fig:awesome_image2}
\endminipage\hfill
\minipage{0.3333\textwidth}
\subcaption{$\sigma=0.1, J=200$.}\vspace{-8pt}
  \includegraphics[width=\linewidth]{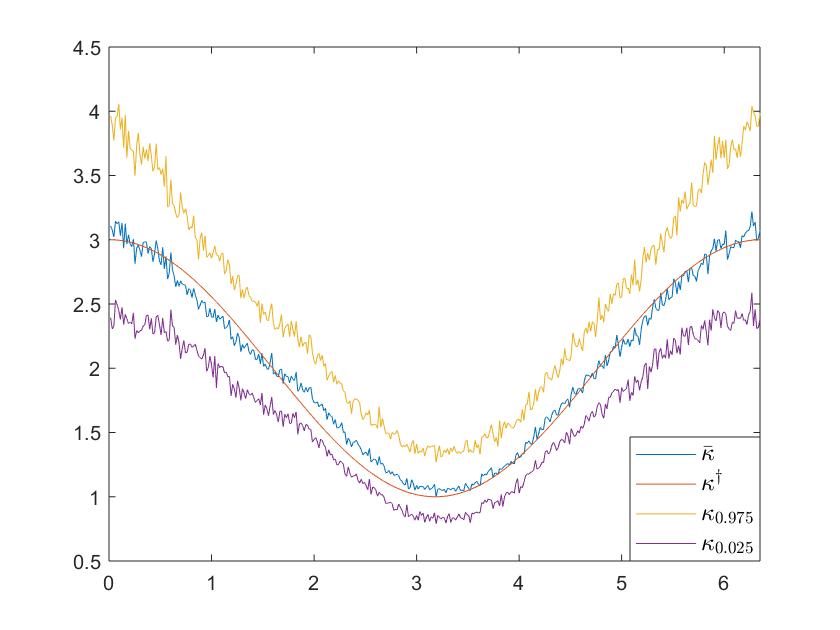}
\label{fig:awesome_image2}
\endminipage\hfill
\minipage{0.3333\textwidth}
\subcaption{$\sigma=0.1, J=400$.}\vspace{-8pt}
  \includegraphics[width=\linewidth]{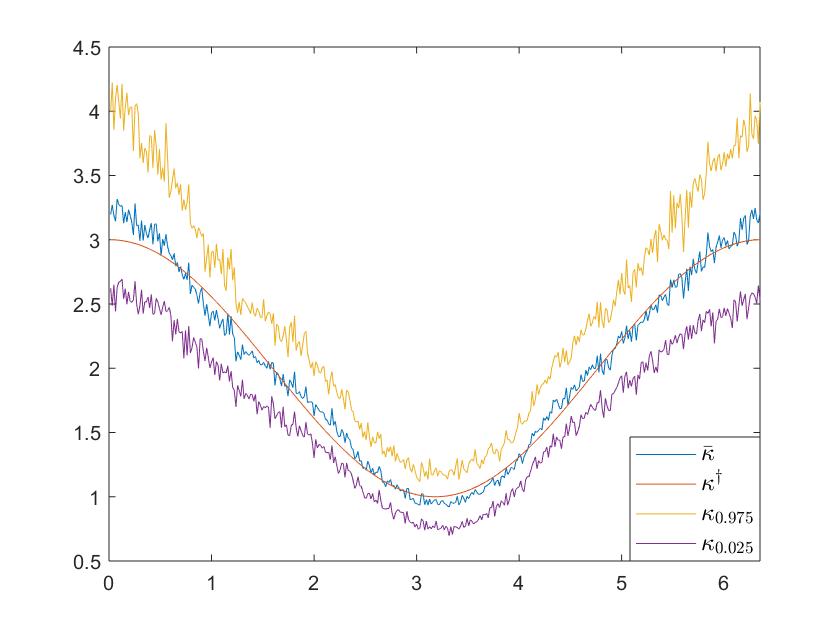}
\label{fig:awesome_image2}
\endminipage
\endminipage\hfill
\vspace{-20pt}
\caption{Posterior means and 95\% credible intervals for $\sigma=0.1$ and different $J$'s. Here $\kappa_{0.025}$ and $\kappa_{0.975}$ represent the 2.5\% and 97.5\% posterior quantiles respectively. }
\label{figure:roughE}
\end{figure}
\begin{remark}\label{rmk:illpose}
We note that the inexact reconstruction is partially due to the ill-posedness of the elliptic inverse problem \cite{mueller2012linear}. As can be seen in Table \ref{table:smoothE}, the reconstruction error for $\bar{\kappa}$ is much larger than that for $\bar{u}$: a wide range of $\kappa$'s around $\kappa^{\dagger}$ give solutions $u$ which are ``close enough'' to $u^{\dagger}$ (within a range of order $\sigma$) that the algorithm cannot distinguish.  When $\sigma$ is large, such tolerance is larger and the inverse problem becomes more difficult. This issue, together with Remark \ref{rmk:limitprior}, explains why one cannot expect exact recovery of $\kappa^{\dagger}$ as seen in Figure \ref{figure:smoothE}.  
\end{remark}

\subsection{Two-Dimensional Elliptic Problem on an Unknown Torus}\label{ssec:ex2}
In this subsection we take $\mathcal{M}$ to be $\mathbb{T}^2$ embedded in $\mathbb{R}^3$ through 
\begin{align}
\iota(\omega_1,\omega_2) = \big((2+\cos \omega_1) \cos \omega_2, (2+\cos \omega_1) \sin \omega_2, \sin \omega_1\big)^T, \quad \omega_1,\omega_2 \in [0,2\pi], \label{eq:torus}
\end{align} 
and the Riemannian metric is 
\begin{align*}
g(\omega_1,\omega_2) = 
\begin{bmatrix}
1& 0 \\
0& (2+\cos\omega_1)^2
\end{bmatrix}.
\end{align*}
The truth is set to be
\begin{align*}
\kappa^{\dagger} (\omega_1,\omega_2) =2+ \sin \omega_1 \sin\omega_2,\quad u^{\dagger}=\sin \omega_1 \sin \omega_2,
\end{align*}
and $f$ is again specified through \eqref{eq:RHS}. One can check that both $u^{\dagger}$ and $f$ have mean zero, i.e., $\int u^{\dagger} \sqrt{\text{det}g}=\int f \sqrt{\text{det}g}=0$. 
For computational reasons as in Remark \ref{rmk:compcost}, we generate the point cloud according to \eqref{eq:torus} from a 20$\times$20 uniform grid on $[0,2\pi]\times[0,2\pi]$ and the observations are given as noisy pointwise evaluations at all points. The graph Laplacian for the prior is constructed with $k=16$ neighbors and again we empirically tune the parameters: $\tau=0.08$ and $s=6$. Unlike the ellipse case, we cannot get almost exact spectral approximation given that we are only discretizing each dimension by 20 points. For this reason, we need a large $s$ to ensure sufficient decay of the spectrum.  In Figure \ref{figure:torus}, we plot the posterior means and the standard deviations as functions on $[0,2\pi]\times[0,2\pi]$; we note that the uncertainty is large when the function $\sin\omega_1\sin\omega_2$ crosses 0.  Table \ref{table:torus} quantifies the reconstruction error as usual and the reconstruction error for $u^{\dagger}$ is again much smaller than the noise level, which are $2\%,10\%$, and $20\%$ respectively. However, the reconstruction error for $u^{\dagger}$ decreases with $\sigma$ decreasing while the error for $\kappa^{\dagger}$ does the opposite, a manifestation of the ill-posedness.

\begin{table}[h!]
\centering
\begin{tabular}{ |c|c|c|c| } 
 \hline
$\sigma$ &  0.01 &     0.05  &   0.1 \\ \hline
 $\|\bar{\kappa}-\kappa^{\dagger}\|_{2}/\|\kappa^{\dagger}\|_2$ &8.56\% &8.34\% & 6.94\% \\ 
 \hline
 $\|\bar{u}-u^{\dagger}\|_{2}/\|u^{\dagger}\|_2$ & 1.8\% & 2.8\% &4.8\% \\ \hline
 $\sqrt{n}\sigma/\|u^\dagger\|_2$ & 2\% & 10\% & 20\% \\ \hline
\end{tabular}
\caption{Relative error of $\kappa^{\dagger}$ and $\bar{u}$ for different noise level, $\sigma$'s. In the last row, the relative noise level is reported for diagnostic purpose. Particularly, note that the reconstruction error for $u^\dagger$ is much smaller than the relative noise level.}
\label{table:torus}
\end{table}


\begin{figure}[!htb]
\minipage{1\textwidth}
\minipage{0.3333\textwidth}
\subcaption{$\sigma=0.01$ (mean)}\vspace{-8pt}
  \includegraphics[width=\linewidth]{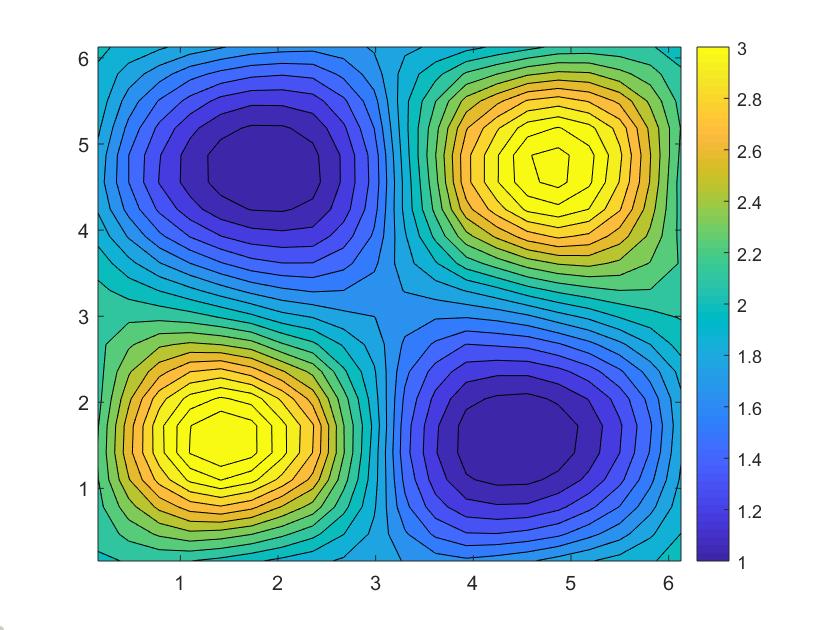}
\label{fig:awesome_image2}
\endminipage\hfill
\minipage{0.3333\textwidth}
\subcaption{$\sigma=0.05$ (mean)}\vspace{-8pt}
  \includegraphics[width=\linewidth]{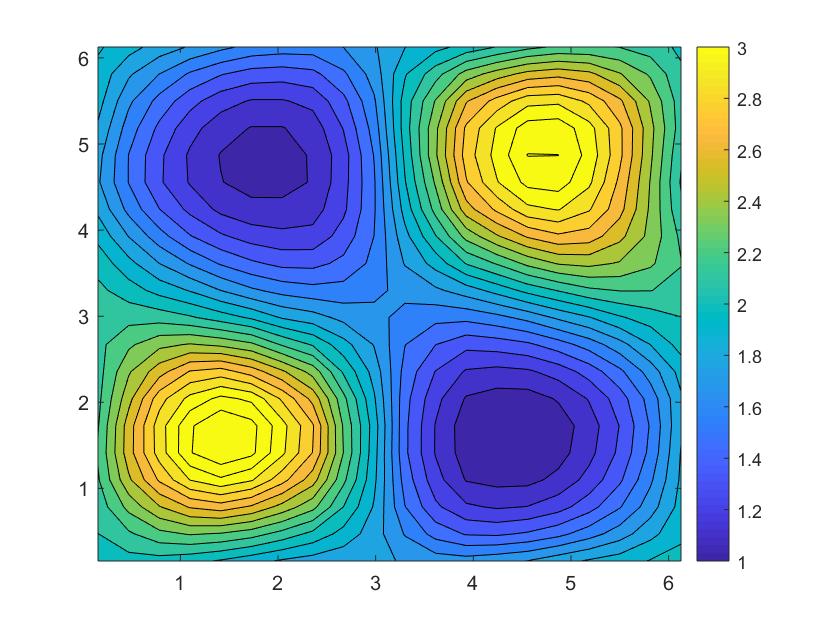}
\label{fig:awesome_image2}
\endminipage\hfill
\minipage{0.3333\textwidth}
\subcaption{$\sigma=0.1$ (mean)}\vspace{-8pt}
  \includegraphics[width=\linewidth]{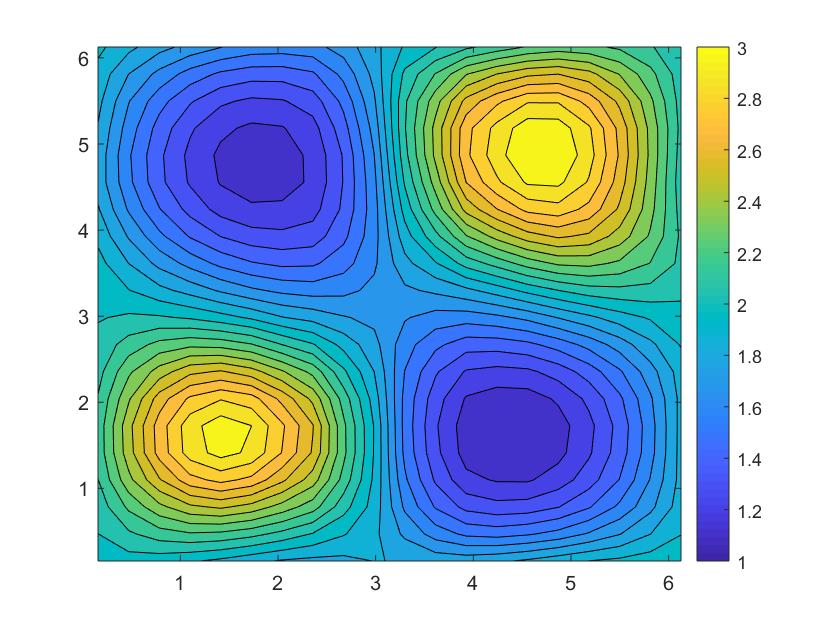}
\label{fig:awesome_image2}
\endminipage
\endminipage\hfill

\minipage{1\textwidth}
\minipage{0.3333\textwidth}
\subcaption{$\sigma=0.01$ (std)}\vspace{-8pt}
  \includegraphics[width=\linewidth]{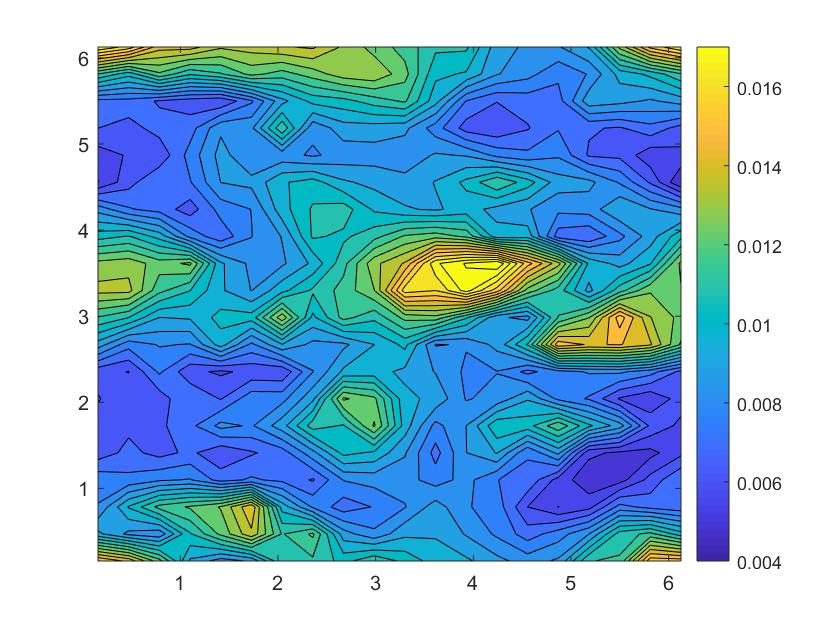}
\label{fig:awesome_image2}
\endminipage\hfill
\minipage{0.3333\textwidth}
\subcaption{$\sigma=0.05$ (std)}\vspace{-8pt}
  \includegraphics[width=\linewidth]{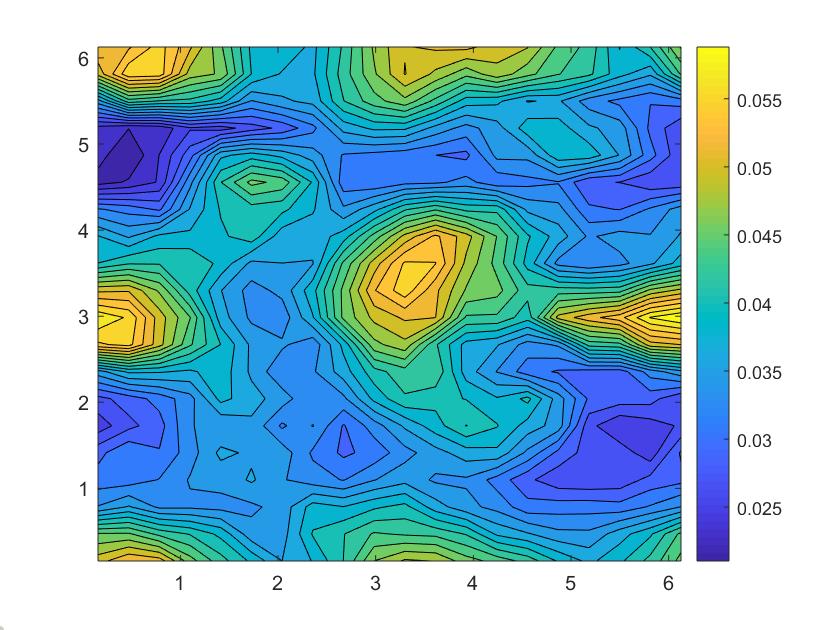}
\label{fig:awesome_image2}
\endminipage\hfill
\minipage{0.3333\textwidth}
\subcaption{$\sigma=0.1$ (std)}\vspace{-8pt}
  \includegraphics[width=\linewidth]{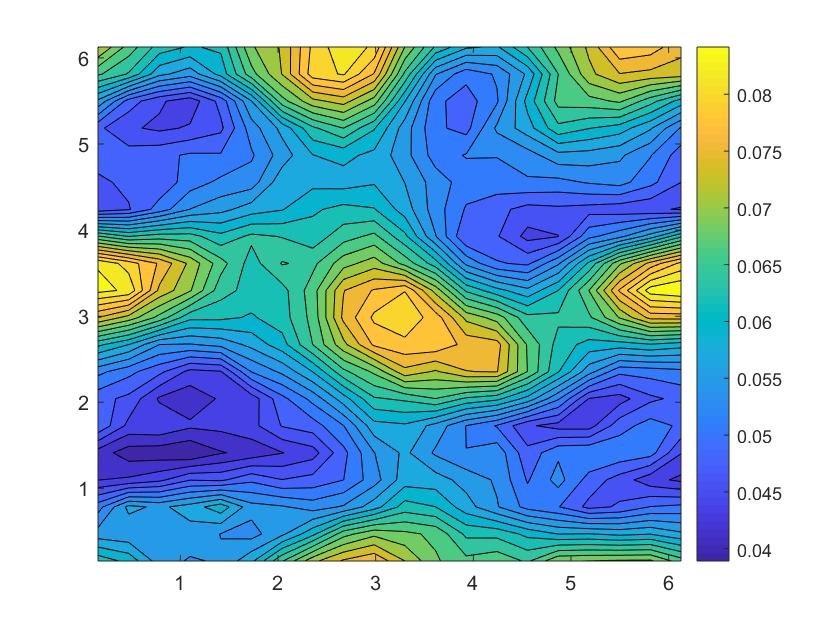}
\label{fig:awesome_image2}
\endminipage
\endminipage\hfill
\vspace{-20pt}
\caption{Posterior means (first row) and standard deviations (second row) of posteriors for different $\sigma$'s. horizontally. We have extended by interpolation the point cloud solution in order to ease visualization.}
\label{figure:torus}
\end{figure}

\begin{figure}[!htb]
\minipage{1\textwidth}
\centering
\minipage{0.4\textwidth}
\subcaption{Posterior mean $\bar{\kappa}$.}\vspace{-7pt}
  \includegraphics[width=\linewidth]{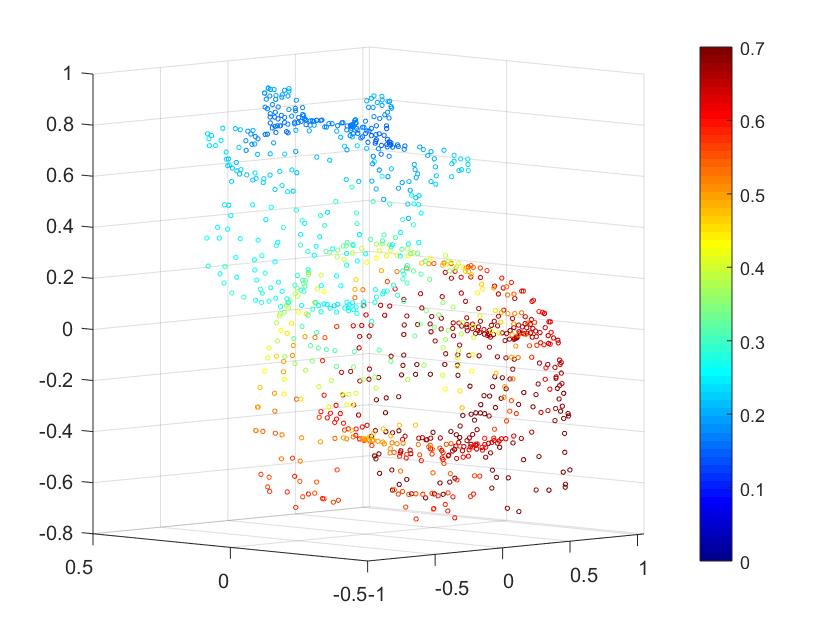}
\label{fig:awesome_image2}
\endminipage
\minipage{0.4\textwidth}
\subcaption{Truth $\kappa^{\dagger}|_X$.}\vspace{-7pt}
  \includegraphics[width=\linewidth]{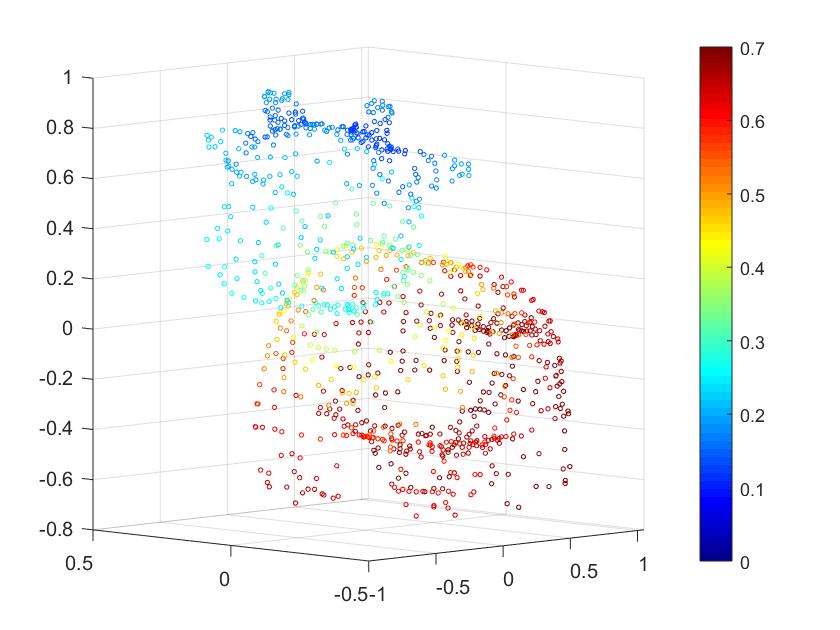}
\label{fig:awesome_image2}
\endminipage
\endminipage\hfill
\minipage{1\textwidth}
\centering
\minipage{0.4\textwidth}
\subcaption{Error  $\left|\bar{\kappa}-\kappa^{\dagger}|_X\right|$.}\vspace{-5pt}
  \includegraphics[width=\linewidth]{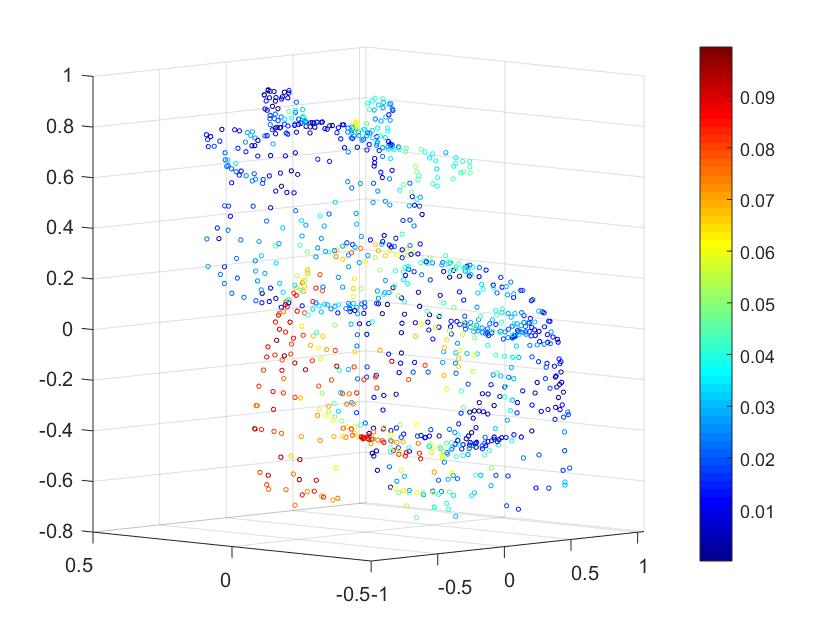}
\label{fig:awesome_image2}
\endminipage
\minipage{0.4\textwidth}
\subcaption{Standard deviation.}\vspace{-7pt}
  \includegraphics[width=\linewidth]{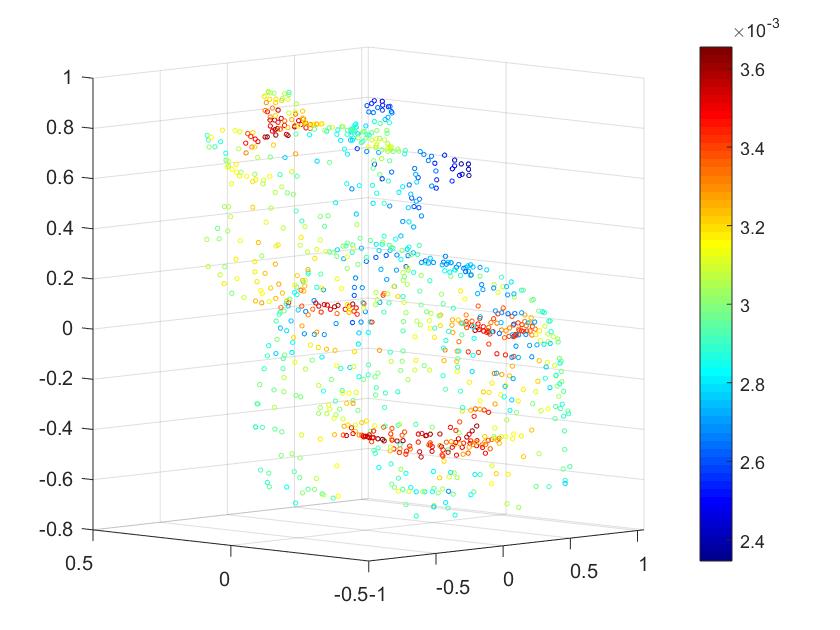}
\label{fig:awesome_image2}
\endminipage
\endminipage\hfill
\vspace{-20pt}
\caption{Reconstruction for the cow-shaped manifold.}
\label{figure:cow}
\end{figure}

\begin{remark}
For this example we solved equation \eqref{eq:pdeM} by taking the pseudoinverse $\hat{u}_n=(L^{\kappa}_{\eps,n})^{\dagger} f_n$ instead, as this is numerically more stable than taking the eigenvalue decomposition of the asymmetric matrix $L^{\kappa}_{\eps,n}$. Moreover, $\hat{u}_n$ is consistent with $u^{\dagger}$ for this specific problem as explained below. We have that $\hat{u}_n$ solves the following problem:
\begin{align}
	\hat{u}_n = \text{min} \left\{\|u\|_2: u \in \operatorname{argmin} \|L^{\kappa}_{\eps,n}u-f_n\|_2 \right\}. \label{eq:svd}
\end{align}
The fact that $f$ has zero mean implies $\langle f_n, 1\rangle_q=0$ in the large $n$ limit. Then equation \eqref{eq:pdeM} has a strong solution as mentioned in Subsection \ref{ssec:graphposterior} and so the characterization \eqref{eq:svd} implies that $\hat{u}_n$ is also a strong solution, with $\sum_{i=1}^n \hat{u}^i_n=0$.  Notice that the truth $u^{\dagger}$ also satisfies $\int u^{\dagger} =0$ and hence makes $\hat{u}_n$ consistent. 
\end{remark}

\subsection{Two-Dimensional Elliptic Problem on an Unknown Artificial Surface}\label{ssec:ex3}
In this subsection we consider an artificial dataset from Keenan Crane’s 3D repository \cite{3dmodel}. The dataset is made of 2930 points sampled from a cow-shaped surface homeomorphic to the two-dimensional sphere. The purpose of this subsection is to demonstrate that our kernel method can be applied to more complicated manifolds. To avoid an inverse crime  \cite{kaipio2006statistical}, we generate the truth using all 2930 points but solve the inverse problem on a subset of size $n=1000$. 

To be more precise, we generate $\kappa^{\dagger}$ from the Gaussian measure $\mathcal{N}(0,(\tau I+\Delta_{2930}))^{-s}$, where $\Delta_{2930}$ is the self-tuning graph Laplacian constructed with $k=100$ neighbors and $\tau=0.7$, $s=6$. We then set $u^{\dagger}$ to be $10(\phi_2-c)$, where $\phi_2$ is the second eigenvector of $\Delta_{2930}$ and $c$ is a constant chosen below. The factor 10 in the definition of $u$ is only to match the order of magnitude with $\kappa^{\dagger}$. We then  set $f=L_{\eps,2930}^{\kappa^{\dagger}} u^{\dagger}$. This would serve as our ground truth, and now we solve the inverse problem on a random subset $X$ of $n=1000$ points. \par
On the point cloud $X$, the truth becomes $\kappa^{\dagger}|_X$ and $u^{\dagger}|_X$. As mentioned above, the solution $u^{\dagger}|_X$ needs to have zero mean to be consistent with our theory. Since we do not have access to the Riemannian metric as in the previous two examples, we instead require $\langle u^{\dagger}|_X, 1 \rangle_q=0$, which gives the choice of $c$ above. Again we consider noisy pointwise observations at all 1000 points.  The inputs of the problem are now $f|_X$ and noisy $u^{\dagger}|_X$. The noise level is 10\%, which gives $\sigma=0.0186$. The prior that we use has the same parameter $\tau$ as used to obtain our synthetic truth, but is defined using a graph Laplacian  on $X$. Namely, we consider 
\begin{align*}
\pi_n=\mathcal{N}(0,(0.7I+\Delta_{1000})^{-6}),
\end{align*}
where $\Delta_{1000}$ is constructed with $k=80$ neighbors.
Figure \ref{figure:cow} shows the plots of the  posterior mean, the truth, the error, and the standard deviation. The reconstruction errors are $\|\bar{\kappa}-\kappa^{\dagger}|_X\|_2/\|\kappa^{\dagger}|_X\|_2=9.73\%$ and $\|\bar{u}-u^{\dagger}|_X\|_2/\|u^{\dagger}|_X\|_2=16.24\%$. We remark that the large relative error for $u^{\dagger}$ is partly due to the fact that the point cloud of size 1000 does not approximate the original one well and in particular 
\begin{align}
	L_{\eps,1000}^{\kappa^{\dagger}|_X} u^{\dagger}|_X \neq f|_X.	\label{eq:cow}
\end{align}
When we solve for $\hat{u}$ in equation \eqref{eq:cow}, i.e., solving  
\begin{align*}
	L_{\eps,1000}^{\kappa^{\dagger}|_X} \hat{u} = f|_X,
\end{align*}
we get $\|\hat{u}-u^{\dagger}|_X\|_2/\|u^{\dagger}|_X\|_2=17.01\%$ and this is the best one can hope for in terms of reconstructing $u^{\dagger}|_X$. Hence we see that the above relative error has already reached the limit of the method.

\subsection{Hierarchical Bayesian Formulation} \label{sec:HB}
As mentioned in Subsection \ref{sec:ctmprior}, the choice of $\tau$ is crucial and would require some prior knowledge of the length-scale of the function to be reconstructed. In this section, we demonstrate how one can learn the parameter $\tau$ together with $\kappa$ through a hierarchical Bayesian approach proposed in \cite{dunlop2017hierarchical}. We emphasize that the hierarchical approach may not be able to find the precise length-scale of the parameter to be reconstructed, and hence should only be applied when little prior knowledge on the length-scale is available. We will only focus on the point-cloud inverse problem as in Subsection \ref{ssec:pointcloudellipticip}. 

We remark that our choice of priors in \eqref{eq:ctmprior} and \eqref{eq:graphprior} differ from the ones used in \cite{dunlop2017hierarchical} by the scaling constants. In the continuum space, the familiy of measures defined by equation \eqref{eq:ctmprior} are mutually singular, which leads to technical difficulties when designing hierarchical methods. However, in the point cloud setting, the family of measures as in \eqref{eq:graphprior} are simply multivariate normal and are equivalent. The formulation in \cite{dunlop2017hierarchical} then carries over.
 
The idea is to consider a joint prior on $(\theta_n,\tau)$ that takes the form
\begin{align*}
	\Pi(\theta_n,\tau)=\pi_0(\tau) \pi(\theta_n|\tau):=\pi_0(\tau)\pi_{\tau}(\theta_n),
\end{align*} 
where $\pi_0$ is a distribution on $\mathbb{R}^+$ and the conditional distribution $\pi_{\tau}$ is taken as in \eqref{eq:graphprior}. Recall that $\pi_{\tau}$ has the form 
\begin{align*}
	\pi_{\tau} = \mathcal{N}(0,C_{\tau,s}^n), \quad C_{\tau,s}^n=c_n(\tau) (\tau I + \Delta_n)^{-s},
\end{align*} 
where $c_n(\tau)$ normalizes the draws so that $u\sim \pi_n$ satisfies $\mathbb{E}|u_n|^2=n$.  
Now we can define the forward map $F_n: \mathbb{R}^n \times \mathbb{R}^+ \rightarrow \mathbb{R}^n$ that associates a pair $(\theta_n,\tau)$ with the unique mean-zero solution $u_n$ of equation \eqref{eq:pdeM}. We notice that $F_n$ is essentially the same as before except the additional $\tau$ that does not play a role in the definition. Denoting $G_n=D_n \circ F_n$, the joint posterior $\Pi^y$ can be written as a change of measure with respect to the prior, 
\begin{align}
	\frac{d\Pi^y}{d\Pi}(\theta_n,\tau) \propto \exp\left(-\frac{1}{2} |y-G_n(\theta_n,\tau)|^2_{\Gamma}\right) .
\end{align}

\subsubsection{Sampling}
To sample from the joint posterior $\theta_n,\tau | y$, we will use a Metropolis within Gibbs sampling scheme by  updating  of $\theta_n|\tau,y$ and $\tau|\theta_n,y$ alternatingly. Sampling of $\theta_n|\tau,y$ reduces to the non-hierarchical setting, where we use the pCN algorithm. Sampling of $\tau|\theta_n,y$ is more delicate since one needs first to make sense of this conditional distribution. Instead of making the presentation too involved, we will only present the algorithm and refer to \cite{dunlop2017hierarchical} for more details. The idea is to use symmetric random-walk Metropolis-Hastings, with acceptance probability to accept the proposal $\tau$, given the current chain value $\gamma$ as,  
\begin{align}
a(\tau,\gamma) = \exp\left(-\frac{1}{2}\left[ H(\tau)-H(\gamma) \right] \right) \frac{\pi_0(\tau)}{\pi_0(\gamma)} \wedge 1, \label{eq:acceptP}
\end{align}
where, 
\begin{align*}
	H(\tau)=\sum_{i=1}^{n} \log \lambda_i(\tau)+ \frac{\langle  \theta, \phi_i \rangle^2}{\lambda_i(\tau)}.
\end{align*}
Here, $\{(\lambda_i(\tau),\phi_i)\}_{i=1}^{n}$ are the eigenvalue-eigenfunction pairs of $C_{\tau,s}^n$. 

From \eqref{eq:acceptP}, we see that the algorithm favors length-scales $\tau$'s that give small values of $H(\tau)$.  As $\tau$ approaches 0, the first eigenvalue of $(\tau I +\Delta_n)^{-s}$ goes to infinity and so the normalizing constant $c_n(\tau)$ approaches 0. However, the eigenvalues of $(\tau I +\Delta_n)^{-s}$, except the first one,  do not change much since $\tau $ is now a much smaller quantity. Hence when multiplied by $c_n(\tau)$, they converge to 0. In other words $\lambda_i(\tau)$ approaches 0 as $\tau $ goes to 0 for $i\geq 2$, which then implies that $\sum_{i=1}^{n} \log \lambda_i(\tau)\rightarrow -\infty$. So the first term in $H$ is minimized at $\tau=0$, while the second term $\langle  \theta, \phi_i \rangle^2/\lambda_i(\tau)$ favors large $\tau$ by a similar argument as above. Then the minimum of $H$ is attained by balancing the two sums, using the coefficients $\langle  \theta, \phi_i \rangle$. Hence the algorithm will give a $\tau$ that is consistent with these coefficients, reflecting the length-scale of $\theta$.

\subsubsection{Numerical Experiments}
\begin{figure}[!htb]
\minipage{1\textwidth}
\minipage{0.3333\textwidth}
  \includegraphics[width=\linewidth]{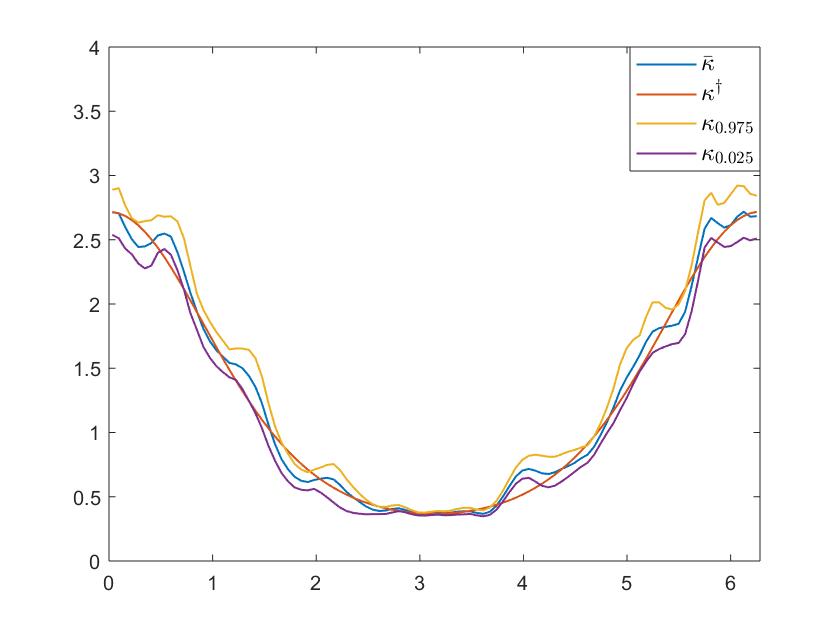}
\label{fig:awesome_image2}
\endminipage\hfill
\minipage{0.3333\textwidth}
  \includegraphics[width=\linewidth]{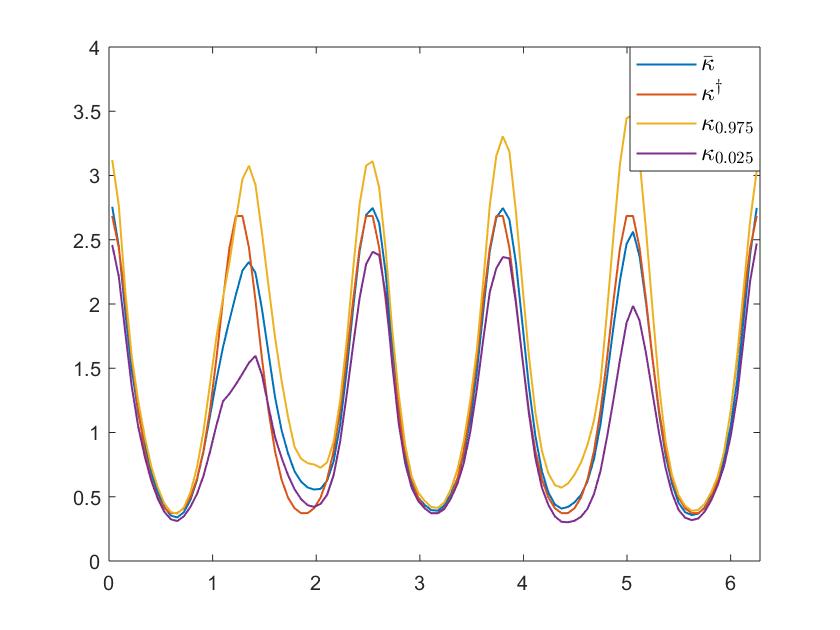}
\label{fig:awesome_image2}
\endminipage\hfill
\minipage{0.3333\textwidth}
  \includegraphics[width=\linewidth]{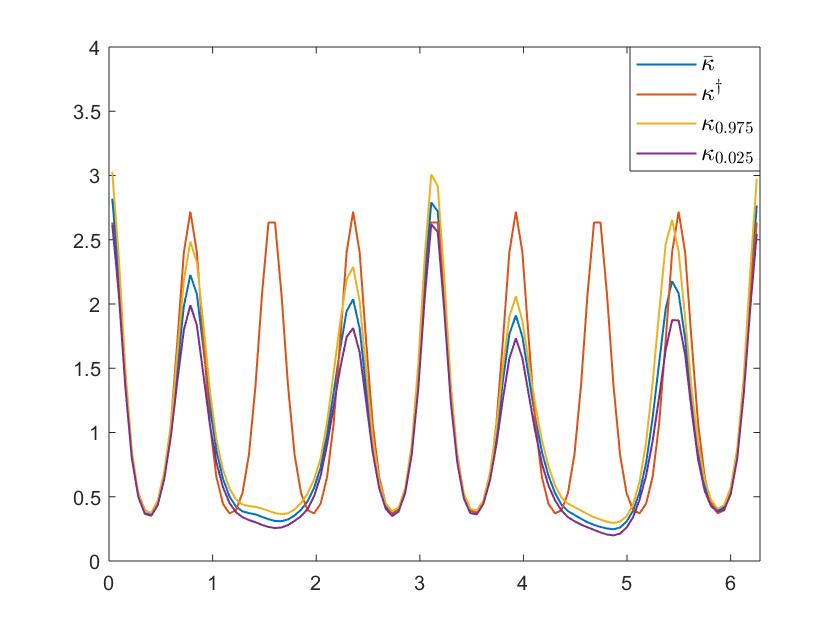}
\label{fig:awesome_image2}
\endminipage
\endminipage\hfill

\minipage{1\textwidth}
\minipage{0.3333\textwidth}
  \includegraphics[width=\linewidth]{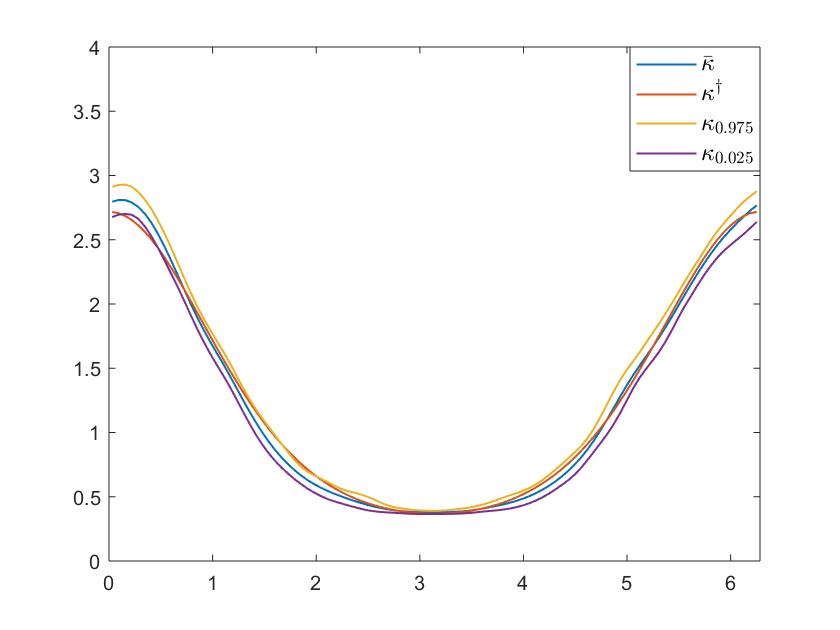}
\label{fig:awesome_image2}
\endminipage\hfill
\minipage{0.3333\textwidth}
  \includegraphics[width=\linewidth]{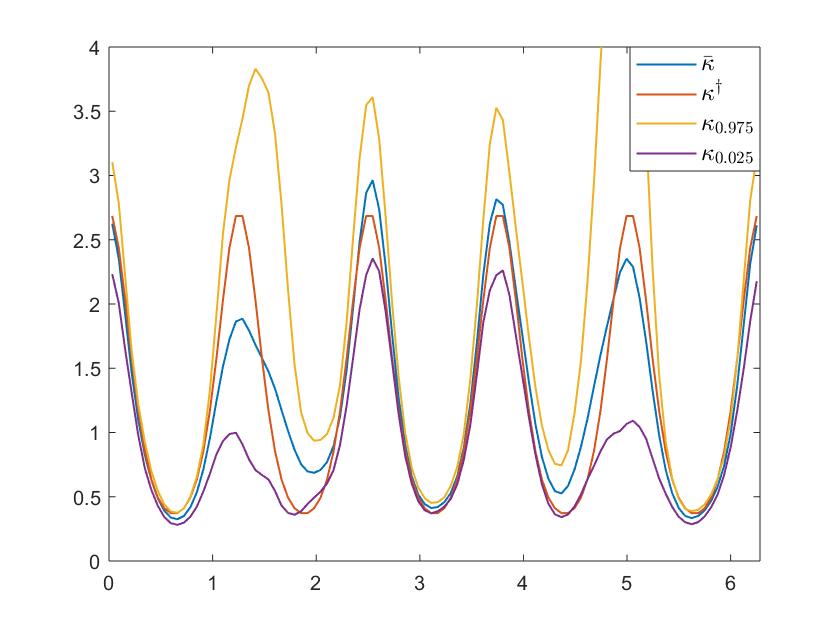}
\label{fig:awesome_image2}
\endminipage\hfill
\minipage{0.3333\textwidth}
  \includegraphics[width=\linewidth]{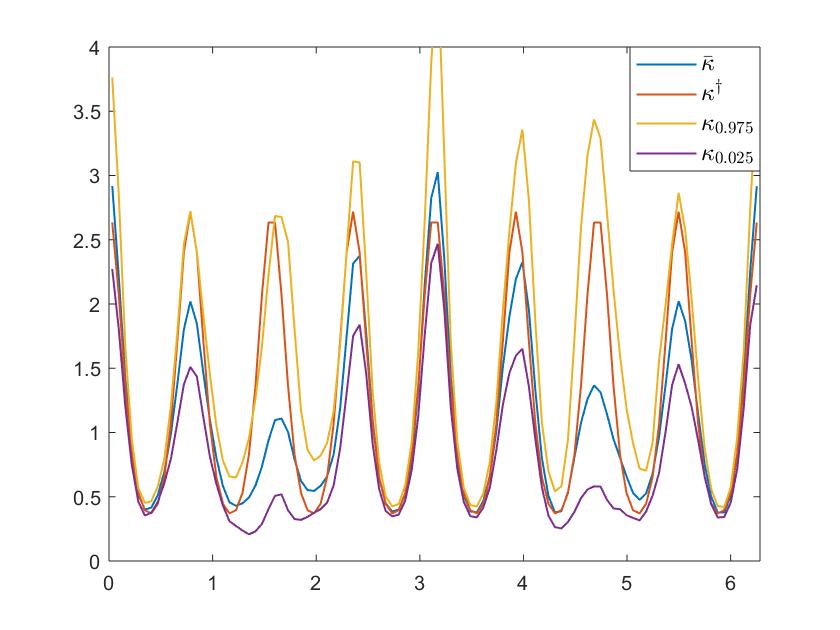}
\label{fig:awesome_image2}
\endminipage
\endminipage\hfill

\minipage{1\textwidth}
\minipage{0.3333\textwidth}
  \includegraphics[width=\linewidth]{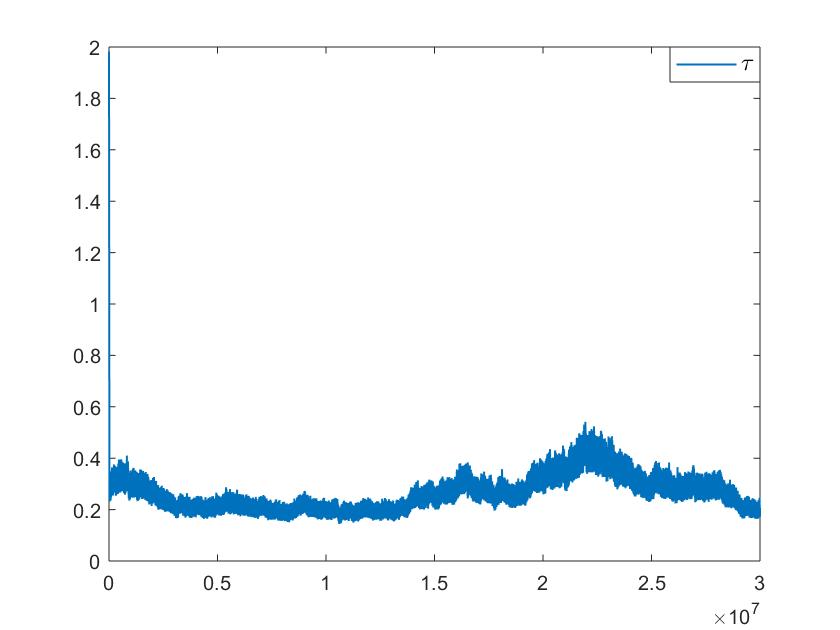}
\label{fig:awesome_image2}
\endminipage\hfill
\minipage{0.3333\textwidth}
  \includegraphics[width=\linewidth]{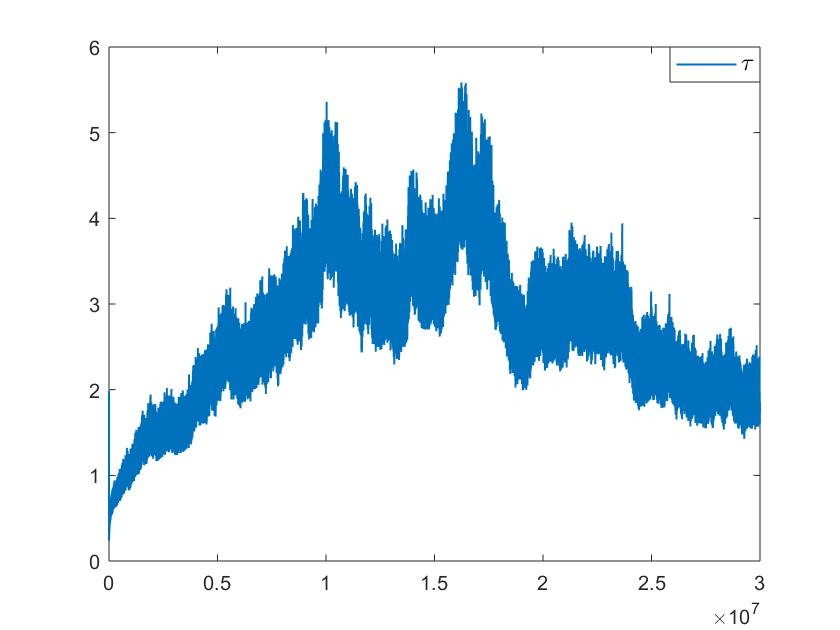}
\label{fig:awesome_image2}
\endminipage\hfill
\minipage{0.3333\textwidth}
  \includegraphics[width=\linewidth]{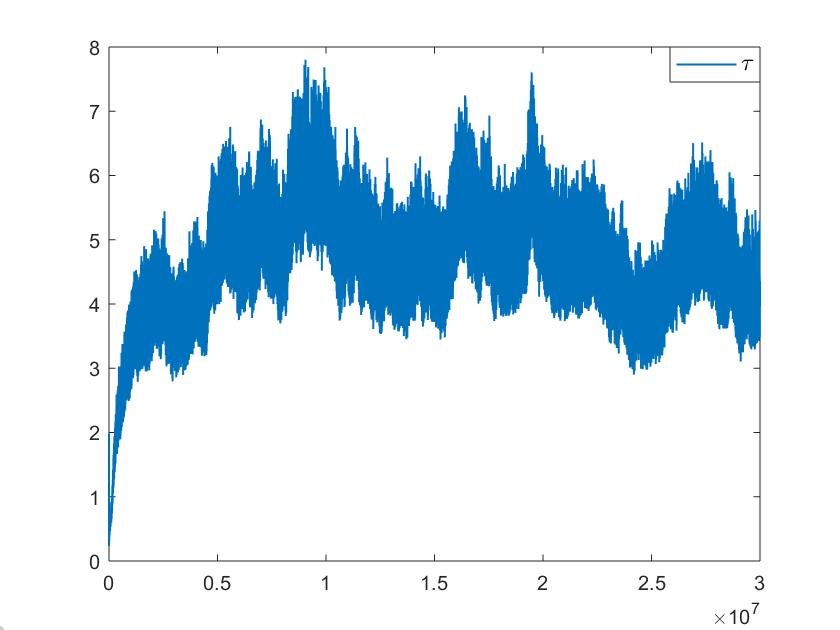}
\label{fig:awesome_image2}
\endminipage
\endminipage\hfill
\caption{Posterior means and 95\%  credible intervals for different truths. Figure are arranged so that the first two rows correspond to non-hierarchical and hierarchical respectively, and the third row shows the sample paths for $\tau$. The three columns represent the truths $e^{\cos(\omega)},e^{\cos(5\omega)},e^{\cos(8\omega)}$ respectively.  }
\label{figure:hier}
\end{figure}

To demonstrate the hierarchical approach, we focus on the ellipse case with three truths of different length-scales: $\kappa^{\dagger}=e^{\cos(\omega)}$, $e^{\cos(5\omega)}$, $e^{\cos(8\omega)}$ with a fixed $f=\frac{1}{5}\sin \omega$. We fix $f$  so that we are solving the same inverse problem, with different  underlying truth $\kappa^{\dagger}$'s. In this case we no longer have analytic solutions for $u$ and we use the MATLAB PDE toolbox . The point cloud is generated as in Subsection \ref{ssec:ex1} with $n=100$, with pointwise observations with noise level $\sigma=0.01$ at all points. We take $\pi_0=\mathcal{N}(2,1)$ and $s=4$. In the hierarchical setting, it takes longer for the chains to mix, where we run the chain for 3$\times 10^7$ iterations and use the last $5\times10^6$ samples for computations.  We compare the hierarchical approach with the non-hierarchical one, where we use a same $\tau$ to reconstruct the three $\kappa^{\dagger}$'s. Figure \ref{figure:hier} below shows the corresponding reconstructions. 

The first row corresponds to the non-hierarchical approach where we fix $\tau=2$, which is finely tuned to match the length-scale of $e^{\cos(5\omega)}$, for all three problems. The reconstruction is then acceptable for $e^{\cos(5\omega)}$ but is poorer for the other two. For $e^{\cos(\omega)}$, the reconstruction still fits the shape of the truth, but since the prior now has a length-scale much larger than that of $e^{\cos(\omega)}$, the reconstruction is oscillatory. On the other hand, the reconstruction of $e^{\cos(8\omega)}$ fails to  capture the shape of the truth. This is because the prior now has a smaller length-scale than that of $e^{\cos(8\omega)}$, so that frequencies high as $\cos(8\omega)$ barely belongs to the prior.  
The second and third row corresponds to the reconstructions of the hierarchical approach and the corresponding sample paths for $\tau$. We see that the credible intervals capture most of the truths and the reconstructions are much better for $e^{\cos(\omega)}$ and $e^{\cos(8\omega)}$ than with the non-hierarchical approach. Table \ref{table:hier} quantifies the reconstruction error. We notice that the hierarchical approach performs worse than the non-hierarchical one for $e^{\cos(5\theta)};$ this is because we have chosen $\tau=2$ agrees with the length-scale of the true diffusion coefficient. This fact suggests that the hierarchical approach only improves the performance when little prior knowledge on the length-scale is known. From the sample paths for $\tau$'s, we see that the chains have large variance and do not concentrate on a particular value. This is due to the ill-posedness of the inverse problem  where $\kappa$'s of different length-scales give equally good reconstruction of $u$, and hence the algorithm cannot distinguish between them. So in general the algorithm may not give the precise length-scale but a possible range of it.

\begin{table}[h!]
\centering
\begin{tabular}{ |c|c|c|c|c|c|c|c|c|c| } 
 \hline
\multicolumn{1}{|c|}{$\kappa^{\dagger}$} & \multicolumn{2}{|c|}{$e^{\cos(\omega)}$} &    \multicolumn{2}{|c|}{ $e^{\cos(5\omega)}$ }  &  \multicolumn{2}{|c|}{$e^{\cos(8\omega)}$} \\
\hline
Method & NH & H & NH & H & NH & H \\
\hline
 $\|\bar{\kappa}-\kappa^{\dagger}\|_{2}/\|\kappa^{\dagger}\|_2$ &9.07\% & 3.69\% &11.17\% & 17.41\% & 39.38\% & 28.82\%\\ 
 \hline
 $\|\bar{u}-u^{\dagger}\|_{2}/\|u^{\dagger}\|_2$ &0.84\%  &0.72\% &0.73\% & 0.69\% & 0.87\% & 0.83\% 
\\ \hline 
  \multicolumn{1}{|c|}{$\sqrt{n}\sigma/\|u^\dagger\|_2$} & \multicolumn{2}{c|}{ 1.87\%} &    \multicolumn{2}{c|}{1.26\%}  &  \multicolumn{2}{c|}{1.27\%}\\
  \hline 
\end{tabular}
\caption{Relative error of $\bar{\kappa}$ and $\bar{u}$ for different truths. Here "NH" and "H" stand for non-hierarchical and hierarchical respectively. In the last row, the relative noise level  for each $\sigma$ is reported for diagnostic purposes.  }
\label{table:hier}
\end{table}


\begin{remark} \label{rmk:largenoise}
Notice that in the above the noise level has been set to be small. When the noise level $\sigma$ is large, the performance of the hierarchical approach may be worse, as shown in Figure \ref{figure:hier2}.  The reason is that the algorithm sees only the noisy data, which is the truth $u^{\dagger}$ perturbed by noise. In other words, the length-scale of the data is corrupted by the noise, which has length-scale converging to 0 ($\tau \rightarrow \infty$) in the large $n$ and $J$ limit if the noise is independent, i.e., $\Gamma=I$. As shown in Figure \ref{figure:hier2}, the chain for $\tau$ oscillates in a wide range of values, suggesting that the data contains little information on this parameter.
\end{remark}

\begin{figure}[!htb]
\minipage{1\textwidth}
\centering
\minipage{0.3333\textwidth}
  \includegraphics[width=\linewidth]{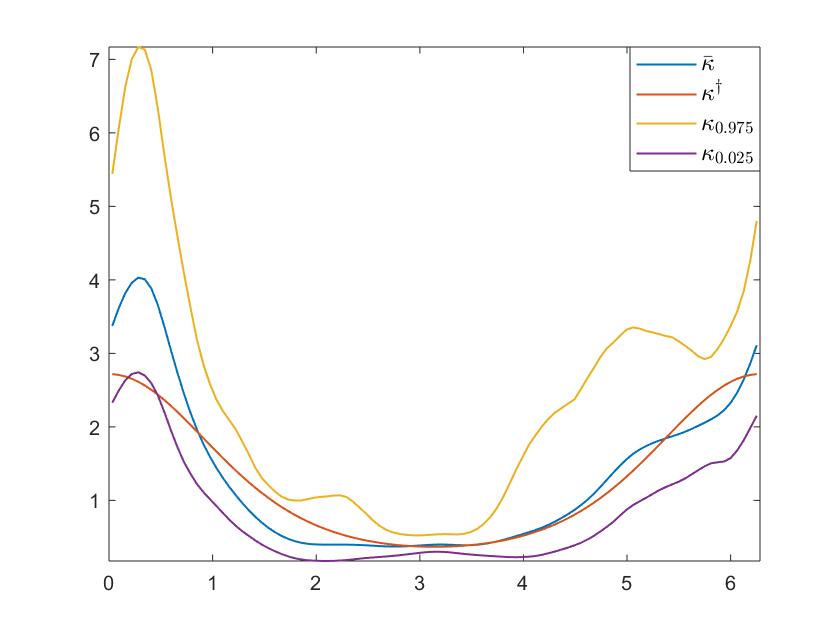}
\label{fig:awesome_image2}
\endminipage
\minipage{0.3333\textwidth}
  \includegraphics[width=\linewidth]{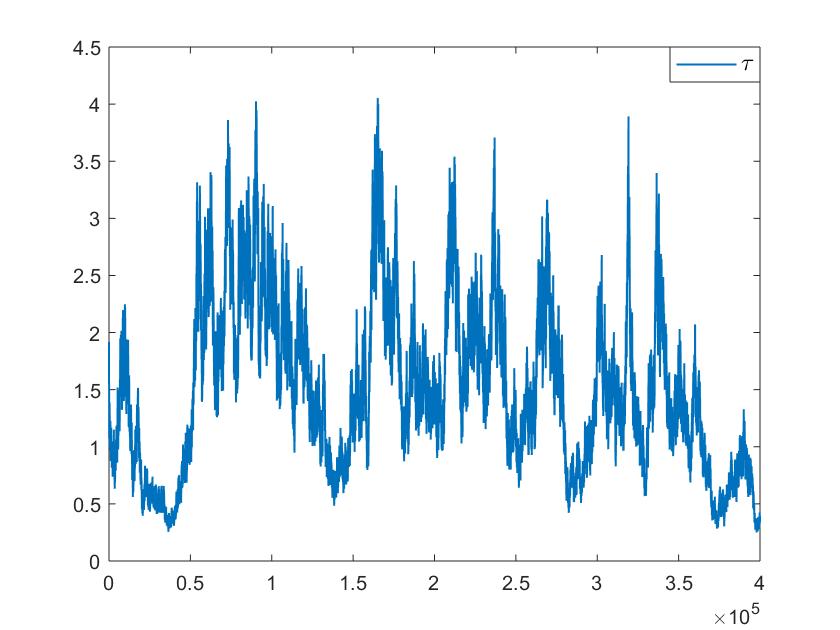}
\label{fig:awesome_image2}
\endminipage
\endminipage\hfill
\caption{Reconsturction of $e^{\cos(\omega)}$ and sample path for $\tau$ when $\sigma=0.1$.  }
\label{figure:hier2}
\end{figure}

\section{Conclusions and Future Work}\label{sec:conclusions}
\begin{itemize}
\item This paper introduced kernel-based methods for the solution of inverse problems on manifolds. We have shown through rigorous analysis that the forward map can be replaced by a kernel approximation while keeping small the total variation distance to the true posterior. Through numerical experiments we have shown that a point cloud discretization to the kernel approximation may allow to implement the inverse problem on point clouds, without reference to the underlying manifold. 
\item An important question of theoretical interest when solving the inverse problem on the point cloud is how to choose optimally the kernel bandwith $\eps$ in terms of the number $n$ of manifold samples. We conjecture that the convergence of the graph posteriors to the ground truth posterior, and guidelines on the choice of kernel bandwith, may be established by generalizing the spectral graph theory results in \cite{burago2013graph,trillos2019local} to anisotropic diffusions, and using the variational techniques introduced in \cite{garcia2018continuum,trillos2017consistency}. The analysis of these questions will be the subject of future work. 
\item We streamlined the presentation by working on a closed manifold with no boundary. We expect that the numerical and theoretical results may be extended to Neumann, Robin, and Dirichlet boundary conditions using the results and ideas in  \cite{gilani2019approximating,li2016convergent,shi2014convergence,thiede2019galerkin}.
\item The practical success of the Bayesian approach is heavily dependent on the choice of prior. Here, we have used Matern-type priors that are flexible models widely used in spatial statistics and the geophysical sciences  \cite{stein2012interpolation}.  While the hierarchical approach to the inverse problem \cite{dunlop2017hierarchical,geoga2019scalable} is effective for learning the prior length-scale from data in certain regimes as we have numerically shown, a more robust algorithm is needed and this merits an extensive further investigation.
\item A topic of further research will be the extension of the kernel-based approximation to PDEs and inverse problems to other PDEs and ODEs beyond the elliptic model considered here. 
\end{itemize}

  \section*{Acknowledgement}
  The research of JH was partially supported by the NSF Grants DMS-1619661 and DMS-1854299. The work of DSA was supported by the NSF Grant DMS-1912818/1912802.
    
\bibliographystyle{abbrvnat}

\bibliography{references}
\appendix
\section{Proofs of Lemmas}
\subsection*{Proof of Lemma \ref{lem:stability}}
The proof proceeds by a standard Lax-Milgram argument. Throughout, $C>0$ denotes a constant independent of $\eps$ and $\kappa$ that may change from line to line.
 Consider the bilinear and linear functionals
\begin{align*}
B: L_0^2 \times L_0^2  &\to \R, \quad  &F: L_0^2 &\to \R, \\
 (u,v) & \mapsto \langle u, \L_\eps^\kappa v \rangle, & v &\mapsto \langle v, f \rangle.
\end{align*}
Clearly, $B$ and $F$ are bounded. To show that $B$ is coercive, note that 
 by \cite{shi2014convergence}[Theorem 7.2] there exists $C>0$ such that, for all $\eps>0$   $v \in L_0^2(\mathcal{M}),$ 
\begin{equation}\label{eq:contraction}
    \langle v,\L_{\eps} v\rangle \geq C \|v\|_{L^2}, \quad 
\end{equation}
where 
\begin{align*}
    \L_\eps v :=\frac{1}{\sqrt{4\pi}\eps^{\frac{m}{2}+1}} \int  \exp\left(-\frac{|x-\tx|^2}{4\eps}\right) [v(x)-v(\tx)] dV(\tx).
\end{align*}
It follows that, for $v \in L_0^2,$ 
\begin{align*}
\langle v, \L^\kappa_{\eps} v \rangle &= \frac{1}{\sqrt{4\pi}\eps^{\frac{m}{2}+1}} \int \exp\left(-\frac{|x-\tx|^2}{4\eps}\right) \sqrt{\kappa(x)\kappa(\tx)} v(x)[v(x)-v(\tx)] dV(\tx) dV(x) \\
&=\frac{1}{\sqrt{4\pi}\eps^{\frac{m}{2}+1}} \int \exp\left(-\frac{|x-\tx|^2}{4\eps}\right) \sqrt{\kappa(x)\kappa(\tx)} v(\tx)[v(\tx)-v(x)] dV(\tx) dV(x) \\
&= \frac{1}{2\sqrt{4\pi}\eps^{\frac{m}{2}+1}} \int \exp\left(-\frac{|x-\tx|^2}{4\eps}\right) \sqrt{\kappa(x)\kappa(\tx)}|v(x)-v(\tx)|^2 dV(\tx) dV(x) \\
&\geq \kmin \frac{1}{2\sqrt{4\pi}\eps^{\frac{m}{2}+1}} \int \exp\left(-\frac{|x-\tx|^2}{4\eps}\right) |v(x)-v(\tx)|^2 dV(\tx) dV(x)\\
&= \kmin \langle v, \L_\eps v \rangle  \geq C \kmin \|v\|^2_{L^2},
\end{align*}
establishing the coercivity of $B.$
The existence and uniqueness of a weak solution, as well as the bound \eqref{eq:boundueps} follow from the Lax Milgram theorem.

\subsection*{Proof of Lemma \ref{lem:operator}}
Our proof follows the same argument as \cite{coifman2006diffusion}[Lemma 8] but keeps track of the coefficients of the higher order terms.  
Let 
\begin{align*}
    G_{\eps}u(x)=\eps^{-\frac{m}{2}} \int h\left(\frac{|x-\tx|^2}{\eps} \right) u(\tx) dV(\tx),\quad h(z)=\frac{1}{\sqrt{4\pi}} e^{-\frac{z}{4}}. 
\end{align*}
Let $0<\beta<\frac{1}{2}.$ We can localize the integration near $x$ due to the exponential decay of $e^{-x^2}$: 
\begin{align*}
    \left|\eps^{-\frac{m}{2}} \int _{\tx\in \mathcal{M}: |\tx-x|>\eps^{\beta}}\exp\left(-\frac{|x-\tx|^2}{4\eps} \right) u(\tx) dV(\tx)  \right|
    &\leq  \eps^{-\frac{m}{4}}\|u\|_{L^2} \sqrt{\eps^{-\frac{m}{2}}  \int _{\tx\in \mathcal{M}: |\tx-x|>\eps^{\beta}}\exp\left(-\frac{|x-\tx|^2}{2\eps} \right)dV(\tx)} \\
    & \leq \eps^{-\frac{m}{4}}\|u\|_{L^2} \sqrt{\int_{x+\sqrt{\eps}\tx \in \mathcal{M}: |\tx|>\eps^{\beta-1/2}  }  \exp \left( -\frac{1}{2}|\tx|^2 \right) dV(\tx)} \\
    & \leq \eps^{-\frac{m}{4}}\|u\|_{L^2} \sqrt{\mathbb{P}\{ \mathcal{N}(0,1) >\eps^{\beta-1/2}  \} } \\
    & \leq \eps^{-\frac{m}{4}}\|u\|_{L^2} \exp\left( -c\eps^{2\beta-1} \right) 
    \leq C\|u\|_{L^2} \eps^{2},
\end{align*}
where in the last inequality since $2\beta<1$, $\exp(-c\eps^{2\beta-1} )$ decays faster than any polynomial in $\eps$ and in particular for $\eps$ small enough it decays faster than $\eps^{2+\frac{m}{4}}$. Therefore, 
\begin{align*}
    G_{\eps}u(x) =\eps^{-\frac{m}{2}} \int_{\tx\in \mathcal{M}: |\tx-x|<\eps^{\beta}} h\left(\frac{|x-\tx|^2}{\eps}\right) u(\tx) dV(\tx) +O(\|u\|_{L^2}\eps^{2}).
\end{align*}
Now we Taylor expand $u$ near $x$. Let $(s_1,\ldots,s_m)$ be the geodesic coordinates at $x$ and $u(\tx)=u(\tx(s_1,\ldots,s_m))=\Tilde{u}(s_1,\ldots,s_m)=\Tilde{u}(s)$. Then 
\begin{align*}
    u(\tx)-u(x)=\Tilde{u}(s)-\Tilde{u}(0)&=\sum_{i=1}^m s_i \frac{\partial \Tilde{u}}{\partial s_i}(0)+\frac{1}{2}\sum_{i=1}^m\sum_{j=1}^m s_is_j \frac{\partial^2 \Tilde{u}}{\partial s_i \partial s_j}(0) 
+ \frac{1}{6} \sum_{i=1}^m\sum_{j=1}^m\sum_{k=1}^m s_is_js_k \frac{\partial^3 \tilde{u}}{\partial s_i \partial s_j \partial s_k}(0) +\delta(s),
\end{align*}
where 
\begin{align*}
        \delta (s)=\frac{1}{24} \int_0^1 \int_0^1 \int_0^1 \int_0^1 \sum_{i=1}^m\sum_{j=1}^m\sum_{k=1}^m\sum_{\ell=1}^m s_is_js_ks_{\ell}\frac{\partial^4 \Tilde{u}}{\partial s_i \partial s_j \partial s_k \partial s_{\ell}}(t_1t_2t_3t_4 s ) dt_1 dt_2dt_3dt_4.
\end{align*}
Then expanding the 4-$th$ order term in $G_{\eps}u$, we have 
\begin{align*}
    |T_u(x)|:=&\left|\eps^{-\frac{m}{2}} \int_{|s|<\eps^{\beta}}  h\left(\frac{|s|^2}{\eps}\right) \delta(s) ds \right| \\
    &\leq \frac{m^4\eps^{4\beta}}{24} \int _0^1\int_0^1\int_0^1\int_0^1 \int_{|s|<\eps^{\beta}} \eps^{-\frac{m}{2}} h\left(\frac{|s^2|}{\eps}\right) \|\nabla^4\Tilde{u}(t_1t_2t_3t_4s)\| ds dt_1dt_2dt_3dt_4\\
    &=\frac{m^4\eps^{4\beta}}{24}  \int _0^1\int_0^1\int_0^1\int_0^1 \int_{|r|<t_1t_2t_3\eps^{\beta}} (t_1t_2t_3t_4)^{-d} \eps^{-\frac{d}{2}} h\left(\frac{|r|^2}{t_1^2t_2^2t_3^2t_4^2\eps}\right)\|\nabla^4\Tilde{u}(r)\| drdt_1dt_2dt_3dt_4\\
    &:=\frac{m^4\eps^{4\beta}}{24}\int _0^1\int_0^1\int_0^1\int_0 \int_{|r|<t_1t_2t_3t_4\eps^{\beta}} K(r) \|\nabla^4\Tilde{u}(r)\| drdt_1dt_2dt_3dt_4.
\end{align*}
By interchanging the order of integration and noticing that $\nabla^4\Tilde{u}(r)=\nabla^4u(x+\xi)$ where $\xi$'s  are directional vectors that are independent of $x$, we have
\begin{align*}
    \|T_u\|^2_{L^2}&\leq \left(\frac{m^4\eps^{4\beta}}{24}\right)^2 \int_{\mathcal{M}} \int_0^1\int_0^1\int_0^1\int_0^1 \left[ \int_{|r|<t_1t_2t_3t_4\eps^{\beta}}  K(r) \|\nabla^4\Tilde{u}(r)\| dr  \right] ^2 dt_1dt_2dt_3dt_4dV(x)\\
    &=\left(\frac{m^4\eps^{4\beta}}{24}\right)^2 \int_{\mathcal{M}} \int_0^1\int_0^1\int_0^1\int_0^1 \left[\int_{|r|<t_1t_2t_3t_4\eps^{\beta}} K(r)dr\right] \left[\int_{|r|<t_1t_2t_3t_4\eps^{\beta}}K(r) \|\nabla^4\Tilde{u}(r)\|^2 dr   \right] dt_1dt_2dt_3dt_4 dV(x)\\
    &=\left(\frac{m^4\eps^{4\beta}}{24}\right)^2 
    \int_0^1\int_0^1\int_0^1\int_0^1\left[\int_{|r|<t_1t_2t_3t_4\eps^{\beta}} K(r)dr\right]  \left[\int_{|r|<t_1t_2t_3t_4\eps^{\beta}}K(r) \int_{\mathcal{M}} \|\nabla^4 u(x+\xi)\|^2  dV(x) dr\right] dt_1dt_2dt_3dt_4 \\
    &=\left(\frac{m^4\eps^{4\beta}}{24}\right)^2 \|\nabla^4 u\|^2_{L^2} \int_0^1\int_0^1\int_0^1\int_0^1\left[\int_{|r|<t_1t_2t_3t_4\eps^{\beta}} K(r)dr\right] ^2dt_1dt_2dt_3dt_4.
\end{align*}
Since the normalizing constant of $K$ is of the right order, its integral can be bounded by a constant that does not depend on $\eps$. So 
\begin{align}
    \|T_u\|_{L^2} \leq C\|\nabla^4 u\|_{L^2} \eps^{4\beta}, \label{eq:boundTL2}
\end{align}
where $C$ is a constant that does not depend on $u$ or $\eps$. 
Now following the same argument as in \cite{coifman2006diffusion} and keeping track of the derivatives of $u$, we have 
\begin{align*}
    G_{\eps}u(x)=u(x)+\eps \left[\omega(x)u(x)+\Delta u(x)\right] +R_u(x), 
\end{align*}
where 
\begin{align*}
    R_u(x)=T_u(x)+O\left(\Big[\|u\|_{L^2}+\|\nabla u(x)\|+\|\nabla^2 u(x)\|+\|\nabla^3u(x)\|\Big] \eps^{2}\right).
\end{align*}
Applying $G_{\eps}$ to $u\sqrt{\kappa}$, we have
\begin{align}
G_{\eps}(u\sqrt{\kappa}) =u\sqrt{\kappa}+ \eps \left[ \omega  u\sqrt{\kappa} + \Delta(u \sqrt{\kappa})\right] + R_{u\sqrt{\kappa}}. \label{eq:Guk}
\end{align}
By expanding the derivatives of $u\sqrt{\kappa}$ and bounding them by the $\infty$-norms of $\kappa$ and its derivatives, we have
\begin{align*}
R_{u\sqrt{\kappa}}(x)=T_{u\sqrt{\kappa}}(x) +O\left(\|\sqrt{\kappa}\|_{\mathcal{C}^4}\Big[ \|u\|_{L^2}+ |u(x)|+\|\nabla u(x)\|+\|\nabla^2 u(x)\|+\|\nabla^3u(x)\| \Big] \eps^2 \right),
\end{align*}
where 
\begin{align*}
\|T_{u\sqrt{\kappa}}\|_{L^2} \leq C\|\sqrt{\kappa}\|_{\mathcal{C}^4}\|\nabla^4 u\|_{L^2} \eps^{4\beta}. 
\end{align*}
By setting $u=1$, we have
\begin{align}
G_{\eps}\sqrt{\kappa}=\sqrt{\kappa} +\eps \left[\omega \sqrt{\kappa} +\Delta\sqrt{\kappa}\right]
+O\left(\|\sqrt{\kappa}\|_{\mathcal{C}^4} \eps^2 \right). \label{eq:Gk}
\end{align}
By combining \eqref{eq:Guk} and \eqref{eq:Gk}, we have
\begin{align*}
\L^\kappa_{\eps}u =\frac{\sqrt{\kappa}}{\eps} \left[uG_{\eps}\sqrt{\kappa}-G_{\eps}(u\sqrt{\kappa}) \right] = \L^\kappa u+O\left(\|\sqrt{\kappa}\|_{\mathcal{C}^4} \eps \right) +\frac{R_{u\sqrt{\kappa}}}{\eps}. 
\end{align*}
Hence it follows that 
\begin{align*}
\|(\L^\kappa_{\eps}-\L^\kappa)u\|_{L^2}\leq C (1 \vee \|u\|_{H^4})\|\sqrt{\kappa}\|_{\mathcal{C}^4}\eps^{4\beta-1}. 
\end{align*}

\subsection*{Proof of Lemma \ref{lem:H4norm}}
First we multiply the equation by $u$ and integrate over $\mathcal{M}$. Integrating by parts, we get 
\begin{align*}
    \int f u = -\int \text{div} (\kappa \nabla u) u =\int \kappa |\nabla u|^2 \geq \kmin \|\nabla u\|^2_{L^2}. 
\end{align*}
By H\"older and Poincar\'e inequalities, there is a constant $C$ that depends only on $\mathcal{M}$ so that 
\begin{align}
    \|\nabla u\|_{L^2} & \leq C \kmin^{-1} \|f\|_{L^2} \leq C \kmin^{-1} \|f\|_{H^3}, \label{eq:deri1}  \\
    \|u\|_{L^2} & \leq C \kmin^{-1} \|f\|_{L^2} \leq C\kmin^{-1} \|f\|_{H^3}. \label{eq:deri0}
\end{align}
Now differentiating the equation with respect to $x_k$ and testing against $u_{x_k}$, we get 
\begin{align*}
    \int f_{x_k} u_{x_k} = -\int \text{div} \left( \frac{\partial }{\partial x_k }(\kappa \nabla u) \right) u_{x_k} 
   = \int \frac{\partial}{ \partial x_k} (\kappa\nabla u) \cdot \nabla u_{x_k} 
    = \int \kappa_{x_k} \nabla u \cdot \nabla u_{x_k} + \kappa \nabla u_{x_k} \cdot \nabla u_{x_k} .
\end{align*}
Using Cauchy's inequality with $\eps$ that $|ab|\leq \eps a^2+\frac{1}{4\eps}b^2 $  we get 
\begin{align*}
    \int f_{x_k}u_{x_k} &\geq -\|\kappa_{x_k}\|_{\infty} \left( \eps \|\nabla u_{x_k}\|^2_{L^2} +\frac{1}{4\eps} \|\nabla u\|^2_{L^2} \right)+\kmin \|\nabla u_{x_k}\|^2_{L^2}\\
    & =\left( \kmin -\eps \|\kappa_{x_k}\|_{\infty}\right) \|\nabla u_{x_k}\|^2_{L^2} -\frac{\|\kappa_{x_k}\|_{\infty}}{4\eps} \|\nabla u\|^2_{L^2}.
\end{align*}
Choosing $\eps=\frac{\kmin}{2(\|\kappa_{x_k}\|_{\infty}+1)}$ and rearranging terms, we get 
\begin{align*}
     \|\nabla u_{x_k}\|^2_{L^2} 
     &\leq \kmin^{-2} \|\kappa_{x_k}\|_{\infty}\big( \|\kappa_{x_k}\|_{\infty}+1 \big) \|\nabla u\|^2_{L^2} +2\kmin^{-1}\int f_{x_k}u_{x_k} \\
     &\leq \kmin^{-2} \|\kappa_{x_k}\|_{\infty}\big( \|\kappa_{x_k}\|_{\infty}+1 \big) \|\nabla u\|^2_{L^2} + \kmin^{-1}\|f_{x_k}\|_{L^2}^2 + \kmin^{-1} \|u_{x_k}\|_{L^2}^2.
\end{align*}
Then we have
\begin{align}
    \|\nabla^2 u\|^2_{L^2} &=\sum_{k=1}^m \|\nabla u_{x_k}\|^2_{L^2} \leq m\kmin^{-2}\big(\|\nabla \kappa\|_{\infty}^2+\|\nabla \kappa\|_{\infty} \big) \|\nabla u\|_{L^2}^2 + \kmin^{-1}\|\nabla f\|_{L^2}^2 + \kmin^{-1} \|\nabla u\|_{L^2}^2\nonumber\\
& \leq C \|f\|^2_{H^3}\left[\kmin^{-3}  + \kmin^{-4} \big(\|\kappa\|_{\mathcal{C}^3}^2+\|\kappa\|_{\mathcal{C}^3}\big)  \right], \label{eq:deri2}
\end{align}
where we have used \eqref{eq:deri1} and $C$ only depends on $\mathcal{M}$. 
Now we bound the norm of the third derivatives by further differentiating the equation with respect to $x_j$ and integrating against $u_{x_kx_j}$. We have again by Cauchy's inequality
\begin{align*}
    \int f_{x_kx_j}u_{x_kx_j}=&-\int \text{div}\left(  \frac{\partial^2}{\partial x_j \partial x_k}(\kappa \nabla u)\right) u_{x_kx_j} \\
    =&\int \frac{\partial^2}{\partial x_j \partial x_k}(\kappa \nabla u)\cdot \nabla u_{x_kx_j}\\
    =& \int \left[ \kappa_{x_kx_j} \nabla u+ \kappa_{x_k} \nabla u_{x_j} +\kappa_{x_j} \nabla u_{x_k}+ \kappa \nabla u_{x_kx_j}  \right]\cdot \nabla u_{x_kx_j} \\
     \geq&   \,
    \kmin \|\nabla u_{x_kx_j}\|^2_{L^2}
    -\|\kappa_{x_kx_j}\|_{\infty}\left( \eps_1 \|\nabla u_{x_kx_j}\|^2_{L^2} +\frac{1}{4\eps_1}\|\nabla  u\|^2_{L^2}\right) \\
    &- \|\kappa_{x_k}\|_{\infty}\left(\eps_2 \|\nabla u_{x_kx_j}\|^2_{L^2} +\frac{1}{4\eps_2} \|\nabla u_{x_j}\|^2_{L^2} \right) 
    - \|\kappa_{x_j}\|_{\infty}\left( \eps_3 \|\nabla u_{x_kx_j}\|^2_{L^2}+\frac{1}{4\eps_3}\|\nabla u_{x_k}\|^2_{L^2} \right)  \\
    =& \,\|\nabla u_{x_kx_j}\|_{L^2}^2 \left(\kmin -\eps_1 \|\kappa_{x_kx_j}\|_{\infty} -\eps_2\|\kappa_{x_k}\|_{\infty}-\eps_3\|\kappa_{x_j}\|_{\infty} \right) \\
     &-\frac{1}{4\eps_1}\|\kappa_{x_kx_j}\|_{\infty} \|\nabla u\|_{L^2}^2 -\frac{1}{4\eps_2} \|\kappa_{x_k}\|_{\infty} \|\nabla u_{x_j}\|_{L^2}^2 -\frac{1}{4\eps_3} \|\kappa_{x_j}\|_{\infty} \|\nabla u_{x_k}\|_{L^2}^2     .
\end{align*}
Now choosing 
\begin{align*}
    \eps_1 =\frac{\kmin}{4(\|\kappa_{x_kx_j}\|_{\infty}+1)}, \quad  
    \eps_2 =\frac{\kmin}{4(\|\kappa_{x_k}\|_{\infty}+1)}, \quad
    \eps_3 =\frac{\kmin}{4(\|\kappa_{x_j}\|_{\infty}+1)}
\end{align*}
and rearranging terms, we get 
\begin{align*}
    \|\nabla u_{x_kx_j}\|^2_{L^2} & \leq  
    2 \kmin^{-1} \left( \|f_{x_kx_j}\|^2_{L^2}+\|u_{x_kx_j}\|^2_{L^2} \right)+4\kmin^{-2} \|\kappa_{x_kx_j}\|_{\infty}\big( \|\kappa_{x_kx_j}\|_{\infty} +1 \big) \|\nabla u\|^2_{L^2} \nonumber  \\
    &+ 4\kmin^{-2} \|\kappa_{x_k}\|_{\infty}\big(\|\kappa_{x_k}\|_{\infty}+1\big) \|\nabla u_{x_j}\|^2_{L^2}  
    +4\kmin^{-2} \|\kappa_{x_j}\|_{\infty}\big(\|\kappa_{x_j}\|_{\infty}+1\big) \|\nabla u_{x_k}\|^2_{L^2}.  
\end{align*}
Then we have 
\begin{align}
    \|\nabla^3 u\|^2_{L^2} =\sum_{j=1}^m \sum_{k=1}^m \|\nabla u_{x_kx_j}\|^2_{L^2}
    &\leq    2\kmin^{-1}\left(\|\nabla^2 f\|^2_{L^2} +\|\nabla^2 u\|^2_{L^2} \right) +4m^2   \kmin^{-2}\big( \|\nabla^2 \kappa\|_{\infty}^2 +\|\nabla^2\kappa\|_{\infty}\big) \|\nabla u\|^2_{L^2} \nonumber\\
    &+ 8m\kmin^{-2}\big( \|\nabla \kappa\|_{\infty}^2+\|\nabla \kappa\|_{\infty}  \big) \|\nabla^2 u\|^2_{L^2} \nonumber \\
   &\leq C\left[  \kmin^{-4}  + \kmin^{-5}\big(\|\kappa\|^2_{\mathcal{C}^3}+\|\kappa\|_{\mathcal{C}^3} \big)+ \kmin^{-6}\big(\|\kappa\|^2_{\mathcal{C}^3}+\|\kappa\|_{\mathcal{C}^3} \big)^2 \right] .  \label{eq:deri3}
\end{align}
Differentiating further and applying a similar argument, gives that
\begin{align}
    \|\nabla^4 u\|^2_{L^2} &\leq  4 \kmin^{-1} \big(\|\nabla^3 f\|^2_{L^2}+\|\nabla^3u\|^2_{L^2}\big) +16m^3 \kmin^{-2} \|\nabla u\|^2_{L^2} \big(\|\nabla^3 \kappa\|^2_{\infty}+\|\nabla^3 \kappa\|_{\infty}\big) \nonumber \\
 &+ 48m^2  \kmin^{-2} \|\nabla^2 u\|^2_{L^2}\big(\|\nabla^2 \kappa\|^2_{\infty}+\|\nabla^2 \kappa\|_{\infty}\big) +48m \kmin^{-2} \|\nabla^3 u\|^2_{L^2} \big(\|\nabla \kappa\|^2_{\infty}+\|\nabla \kappa\|_{\infty}\big) \nonumber\\
&\leq C\|f\|^2_{H^3}\Big[  \kmin^{-5} + \kmin^{-6}\big(\|\kappa\|^2_{\mathcal{C}^3}+\|\kappa\|_{\mathcal{C}^3} \big) 
+ \kmin^{-7}\big(\|\kappa\|^2_{\mathcal{C}^3}+\|\kappa\|_{\mathcal{C}^3} \big)^2+\kmin^{-8}\big(\|\kappa\|^2_{\mathcal{C}^3}+\|\kappa\|_{\mathcal{C}^3} \big)^3\Big]. \label{eq:deri4}
\end{align}
The desired result follows by combining equations \eqref{eq:deri0}, \eqref{eq:deri1}, \eqref{eq:deri2}, \eqref{eq:deri3} and \eqref{eq:deri4}.

\subsection*{Proof of Lemma \ref{lem:Z}}
By Lipschitz continuity of $e^{-x}$ when $x>0$, we have,
\begin{align*}
&\left|\exp\left( -\frac{1}{2} |y-\mathcal{G}_{\eps}(\theta)|_{\Gamma}^2 \right)- \exp\left(-\frac{1}{2} |y-\mathcal{G}(\theta)|^2_{\Gamma}\right) \right| \\
&\leq \left| \frac{1}{2} |y-\mathcal{G}_{\eps}(\theta)|_{\Gamma}^2 -\frac{1}{2} |y-\mathcal{G}(\theta)|^2_{\Gamma} \right|\\
&= \frac{1}{2} \left| (\mathcal{G}(\theta)^T+\mathcal{G}_{\eps}(\theta)^T)\Gamma^{-1}(\mathcal{G}(\theta)-\mathcal{G}_{\eps}(\theta)) + 2y^T \Gamma^{-1} (\mathcal{G}(\theta)-\mathcal{G}_{\eps}(\theta))\right| \\
&\leq \frac{1}{2} \|\Gamma^{-1}\|_2 \Big(|\mathcal{G}(\theta)|+|\mathcal{G}_{\eps}(\theta)|\Big)|\mathcal{G}(\theta)-\mathcal{G}_{\eps}(\theta)|+\|\Gamma^{-1}\|_2|y||\mathcal{G}(\theta)-\mathcal{G}_{\eps}(\theta)|,
\end{align*}
where $\|\Gamma^{-1}\|_2$ is the operator 2-norm of $\Gamma^{-1}$.
By Theorem \ref{thm:1}, Lemma \ref{lem:stability} and \eqref{eq:deri0},
\begin{align*}
	|\mathcal{G}(\theta) -\mathcal{G}_{\eps} (\theta)| &\leq \sqrt{\sum \|\ell_j\|^2} \|u-u_{\eps}\|_{L^2} \leq C\sqrt{\sum \|\ell_j\|^2} A(\theta)\|f\|_{H^3} \eps^{4\beta-1} \\
	|\mathcal{G}(\theta)| +|\mathcal{G}_{\eps} (\theta)| &\leq \sqrt{\sum \|\ell_j\|^2}\big(\|u\|_{L^2}+\|u_{\eps}\|_{L^2}\big)\leq C\sqrt{\sum \|\ell_j\|^2} e^{\|\theta\|_{\infty}} \|f\|_{L^2},
\end{align*}
where $u$ and $u_{\eps}$ are the zero-mean solutions associated with  $\theta$, and $\|\ell_j\|$ is the operator norm of $\ell_j$. Here we have written $A$ as a function of $\theta$ and use the fact that $\kappa_{\text{min}}\geq e^{-\|\theta\|_{\infty}}$. 
So 
\begin{align*}
&\int \left| \exp\left( -\frac{1}{2} |y-\mathcal{G}_{\eps}(\theta)|_{\Gamma}^2 \right)- \exp\left(-\frac{1}{2} |y-\mathcal{G}(\theta)|^2_{\Gamma}\right)\right|  d\pi(\theta)   \\
&\leq C \|\Gamma^{-1}\|_2 \eps^{4\beta-1}\left[ \sum \|\ell_j\|^2 \|f\|_{H^3}^2  \int e^{\|\theta\|_{\infty}}A(\theta) d\pi(\theta) + \sqrt{\sum \|\ell_j\|^2} \|y\|\|f\|_{H^3} \int A(\theta) d\pi(\theta)  \right].
\end{align*}
It now suffices to show $\int \left( e^{\|\theta\|_{\infty}} \vee 1\right) A(\theta) d\pi(\theta)<\infty$. Since $\kappa=e^{\theta}$, we have 
\begin{align*}
    \|\kappa\|_{C^4} \leq C e^{\|\theta\|_{\infty}} \left(\|\theta\|_{\mathcal{C}^4}+\|\theta\|_{\mathcal{C}^4}^2+\|\theta\|_{\mathcal{C}^4}^3+\|\theta\|_{\mathcal{C}^4}^4\right), 
\end{align*}
where $C$ is a constant depending on the dimension $m$. Keeping only the highest order term in $e^{\|\theta\|_{\infty}}$, we have 
\begin{align*}
	A(\theta)& \leq C \left[1 \vee \sqrt{P_1(\|\theta\|_{\mathcal{C}^4}) e^{14 \|\theta\|_{\infty}}}  \right] P_2(\|\theta\|_{\mathcal{C}^4}) e^{\|\theta\|_{\infty}}
& \leq C \left[1 \vee \sqrt{P_1(\|\theta\|_{\mathcal{C}^4}) e^{14 \|\theta\|_{\mathcal{C}^4}}}  \right] P_2(\|\theta\|_{\mathcal{C}^4}) e^{\|\theta\|_{\mathcal{C}^4}}, 
\end{align*}
where $P_1$ and $P_2$  are polynomials. Since $\pi$ is a Gaussian measure on $\mathcal{C}^4$, by Fernique's theorem, 
\begin{align*}
	\int \left( e^{\|\theta\|_{\infty}} \vee 1\right) A(\theta) d\pi(\theta)<\infty.
\end{align*}
It follows that 
\begin{align*}
	\int \left|\exp\left( -\frac{1}{2} |y-\mathcal{G}_{\eps}(\theta)|_{\Gamma}^2 \right)- \exp\left(-\frac{1}{2} |y-\mathcal{G}(\theta)|^2_{\Gamma}\right) \right| d\pi(\theta) \leq C\eps^{4\beta-1},
\end{align*}
where $C$ depends on $\Gamma, y, f$ and the $\ell_j$'s, but is independent of $\eps$.

\end{document}